\documentclass[12pt]{article}
\usepackage{amsmath,amssymb,amsfonts,amsthm}
\usepackage[all]{xy}

\newcounter{item}[section]
\newcounter{kirshr}
\newcounter{kirsha}
\newcounter{kirshb}
\newenvironment{enumroman}{\setcounter{kirshr}{1}
\begin{list}{(\roman{kirshr})}{\usecounter{kirshr}} }{\end{list}}
\newenvironment{enumarab}{\setcounter{kirshb}{1}
\begin{list}{(\arabic{kirshb})}{\usecounter{kirshb}} }{\end{list}}
\newtheorem{theorem}{Theorem}[section]

\newtheorem{lemma}[theorem]{Lemma}
\newtheorem{corollary}[theorem]{Corollary}
\newenvironment{demo}[1]{\noindent{\bf #1.}\upshape\mdseries}
{\nopagebreak{\hfill\rule{2mm}{2mm}\nopagebreak}\par\normalfont}
\theoremstyle{definition}

\newtheorem{example}[theorem]{Example}
\newtheorem{definition}[theorem]{Definition}

\def\C{{\mathfrak{C}}}

\def\At{{\sf At}}

\def\Nr{{\mathfrak{Nr}}}
\def\Fr{{\mathfrak{Fr}}}
\def\Sg{{\mathfrak{Sg}}}

\def\A{{\mathfrak{A}}}
\def\B{{\mathfrak{B}}}
\def\C{{\mathfrak{C}}}
\def\D{{\mathfrak{D}}}
\def\M{{\mathfrak{M}}}

\def\CA{{\bf CA}}

\def\Dc{{\bf Dc}}

\def\K{{\bf K}}
\def\K{{\bf K}}

\def\Rd{{\mathfrak{Rd}}}
\def\(R)RA{{\bf (R)RA}}

\def\Dc{{\bf Dc}}

\def\Dc{{\bf Dc}}

\def\Sc{{\bf Sc}}

 \def\CA{{\sf CA}}
\def\B{{\sf B}}

\def\K{{\sf K}}
 
\def\Nr{{\mathfrak{Nr}}}

\def\Nr{{\mathfrak{Nr}}}
\def\Tm{{\mathfrak{Tm}}}
\def\A{{\mathfrak{A}}}
\def\B{{\mathfrak{B}}}
\def\C{{\mathfrak{C}}}
\def\D{{\mathfrak{D}}}

\def\CA{{\bf CA}}

\def\L{{\mathfrak{L}}}

\def\Nr{{\mathfrak{Nr}}}
\def\Fr{{\mathfrak{Fr}}}
\def\Sg{{\mathfrak{Sg}}}

\def\Rd{{\mathfrak{Rd}}}
\def\Ig{{\mathfrak{Ig}}}
\def\CA{{\bf CA}}

\def\K{{\bf K}}
\def\L{{\bf L}}

\def\(R)RA{{\bf (R)RA}}

\def\Dc{{\bf Dc}}

\def\Dc{{\bf Dc}}

\def\Nr{{\mathfrak{Nr}}}
\def\Fr{{\mathfrak{Fr}}}
\def\Sg{{\mathfrak{Sg}}}

\def\Rd{{\mathfrak{Rd}}}
\def\Ig{{\mathfrak{Ig}}}
\def\CA{{\bf CA}}

\def\K{{\bf K}}
\def\L{{\bf L}}

\def\(R)RA{{\bf (R)RA}}

\def\Dc{{\bf Dc}}

\def\Dc{{\bf Dc}}

\def\Dc{{\bf Dc}}

\def\Sc{{\bf Sc}}

 \def\CA{{\sf CA}}

\def\M{{\mathfrak{M}}}

\def\At{{\mathfrak{At}}}

\def\K{{\bf K}}

\def\A{{\mathfrak{A}}}
\def\B{{\mathfrak{B}}}
\def\C{{\mathfrak{C}}}
\def\D{{\mathfrak{D}}}

\def\Nr{{\mathfrak{Nr}}}

\def\CA{{\bf CA}}

\def\Sg{{\mathfrak Sg}}

\def\At{{\sf At}}

\def\rng{{\sf rng}}
\def\dom{{\sf dom}}
\def\Fl{{\mathfrak{Fl}}}

\def\Co{{\sf Co}}

\def\Dc{{\sf Dc}}
\def\Sc{{\sf Sc}}
\def\Ss{{\sf Ss}}

\def\TCA{{\sf TCA}}
\def\CA{{\sf CA}}
\def\TDc{{\sf TDc}}

\def\TPCA{{\sf TPCA}}
\def\L{{\mathfrak{L}}}

\title {Amalgamation, interpolation and congruence extension properties in topological cylindric algebras}
\author{Tarek Sayed Ahmed}

\begin{document}
\maketitle
\begin{abstract}
\noindent  
We study  various amalgamation properties in topological cylindric algebras of all dimensions.
\footnote{Topological logic, Chang modal logic, cylindric algebras, representation theory, amalgamation, congruence extension, interpolation
Mathematics subject classification: 03B50, 03B52, 03G15.}
\end{abstract}

\section{Amalgamation; positive results}

We turn to investigating amalgamation properties for various subclasses of $\sf TCA_{\alpha}$, most, but not all,
consisting solely of representable algebras. We shall adress all dimensions.
All our positive results in the infinite dimensional case 
will follow from the interpolation result proved in part 1, which we recall: 

\begin{theorem}\label{in} Let $\alpha$ be an infinite ordinal. let $\beta$ be a cardinal.  Let $\rho:\beta\to \wp(\alpha)$ such that
$\alpha\sim \rho(i)$ is infinite for all $i\in \beta$. Then $\Fr_{\beta}^{\rho}\TCA_{\alpha}$ has the interpolation property.
\end{theorem}
This theorem  will lead to theorem \ref{s} showing that a certain class of algebras
has the super amalgamation property, which is the main source for all 
positive results obtained.  

All negative results (for both finite and infinite dimension) are obtained by bouncing them to the $\CA$ case, using the earlier 
observation made, namely, that {\it any $\sf CA_{\alpha}$ can be expanded to a $\sf TCA_{\alpha}$ such that
the latter is representable if and only if the former is.} We start with the relevant definitions:

\begin{definition} Let $\sf L$ be a class of algebras.
\begin{enumroman}
\item  $\A_0\in \sf L$ is in the {\it amalgamation base} of $\sf L$ if for all $\A_1, \A_2\in \sf L$ and injective homomorphisms
$i_1:\A_0\to \A_1$ $i_2:\A_0\to \A_2$
there exist $\D\in \sf L$
and injective homomorphisms $m_1:\A_1\to \D$ and $m_2:\A_2\to \D$ such that $m_1\circ i_1=m_2\circ i_2$.
In this case, we say that $\D$ is {\it an amalgam} of $\A_1$ and $\A_2$ over $\A_0$, via $m_1$ and $m_2$, or simply an amalgam.

\item Let everything be as in (i). If, in addition, we have  $m_1\circ i_1(\A_0)=m_1(\A_1)\cap m_2(\A_2)$ then
$\A_0$ is said to be in the {\it strong amalgamation base} of $\sf L$, and like in the previous item
$\D$ is called a {\it strong amalgam}.

\item Let everything be as in (i) and assume that the algebras considered are endowed with a partial order.
If in addition, $(\forall x\in A_j)(\forall y\in A_k)
(m_j(x)\leq m_k(y)\implies (\exists z\in A_0))(x\leq i_j(z)\land i_k(z) \leq y))$
where $\{j,k\}=\{1,2\}$, then we say that $\A_0$ lies in the {\it super amalgamation base} of $\sf L$, and
$\D$ is called {\it a super amalgam}.

\item $\sf L$ has the {\it amalgamation property}, if the amalgam base of $\sf L$ coincides with $\sf L$.
Same for strong amalgamation and super amalgamation.
\end{enumroman}
\end{definition}
We write $AP$, $SAP$ and $SUPAP$ for the amalgamation, strong amalgamation and super amalgamation properties, respectively.
We write ${\sf APbase(K)}$, ${\sf SAPbase(K)}$, and ${\sf SUPAPbase(K)}$, for the amalgamation, 
strong amalgamation, and super amalgamation base
of the class $\sf K$, respectively.
Notice that $SUPAP$ also implies $SAP$ by writing 
the extra condition for $SAP$
as follows: 
$$(\forall x\in A_1)(\forall y\in A_2)
[m(x)=n(y)\implies (\exists z\in A_0)(x=f(z) \land y=h(z))].$$

We will sometimes seek amalgams, and for that matter strong or super amalgams, 
for algebras in a certain class in a possibly bigger one.

\begin{definition} Let $\sf K\subseteq \sf L$.
We say that $\sf K$ has {\it $AP$ with respect to ${\sf L}$} if amalgams can always be found
in $\sf L$. More precisely, for any $\A,\B,\C\in \sf K$ and any injective homomorphisms
$f:\A\to \B$ and $g:\A\to \C$ then there exist $\D\in \sf L$ and injective homomorphisms
$m:\B\to \D$ and $n:\C\to \D$ such that $m\circ f=n\circ g$. 
The analogous definition applies equally well when we replace
$AP$ by $SAP$ or $SUPAP$.
\end{definition}
The next property is different than the amalgamation property; we do not require that both homomorphisms from the base
algebra are injective; 
only one of them is. The definition is taken from \cite{george}.

More precisely:
\begin{definition} A class  $\sf L$ has the {\it transferable injections property}, or $TIP$ for short, 
if  for all $\A_0, \A_1, \A_2\in \sf L$ and injective homomorphism
$i_1:\A_0\to \A_1$, and homomorphism  $i_2:\A_0\to \A_2$
there exist $\D\in \sf L$
and an injective homomorphism $m_1:\A_1\to \D$ 
and a homomorphism $m_2:\A_2\to \D$ such that $m_1\circ i_1=m_2\circ i_2$. We say that $\D$ is a $TI$ amalgam.
\end{definition}

From now on $\alpha$ is an arbitrary ordinal $>0$.
\begin{definition}
Let $\A\in \sf TCA_{\alpha}$. Then a filter $F$ of $\A$ is a Boolean filter, that satisfies in addition that ${\sf q}_ix=x$ for every $i<\alpha$
and every $x\in A$.
\end{definition}
The following lemma is crucial for our later algebraic manipulations.
We show that filters so defined correspond to congruences, thus filters and congruences can be treated equally giving 
quotient algebras.

\begin{theorem}\label{lattice} Let $\A\in \sf TCA_{\alpha}$. Let ${\sf Filt}(\A)$ be
the lattice of filters (with inclusion) on $\A$, and $\Co(\A)$be the lattice of congruences on $\A$.
Then ${\sf Filt}(\A)\cong \Co(\A)$. Furthermore, $\Theta$ restricted to maximal filters is an isomorphism into the set of maximal congruences.
\end{theorem}
\begin{proof}
The map $\Theta: \Co(\A)\to {\sf Filt}({\A})$, defined via
$$\cong\mapsto \{x\in \A: x\cong 1\},$$
is an isomorphism,
with inverse $\Theta^{-1}: {\sf Filt}(\A)\to \Co(\A)$
defined by
$$F\mapsto R=\{(a,b)\in A\times A: a\oplus b\in F\},$$
where $\oplus$ as before denotes the symmetric diference.

Indeed, let $\cong$ be a congruence on $\A$. Then we show that $F=\{a\in A: a\cong 1\}$ is a filter. 
Let $a, b\in F$. Then $a\cong 1$ and $b\cong 1$, hence $a\cdot b\cong 1$, so that $a\cdot b\in F$. 
Let $a\in F$ and $a\leq b$. Then $a+b=b$ and we obtain
$b=a+ b\cong 1+b=1$. Hence $b\in F$. Now finally, assume that
$a\in F$ and $i<\alpha$. T
hen $a\cong 1$ so ${\sf q}_ia\cong {\sf q}_i1=1$, and we are done.

Conversely, let $F$ be a filter and let $\cong_F$ be given by $a\oplus b\in F$. 
Then it is straightforward to see that $\cong _F$ is a congruence with
respect to the Boolean  operations, cylindrifiers and substitutions.
It remains to check that $\cong_F$ is a congruence with respect 
to the interior operators. Let $i<\alpha$.
If $a\cong_F b$ then, by definition,  $a\oplus  b\in F,$ hence ${\sf q}_i(a\oplus  b)\in F$
be the definition of $F$. 
But ${\sf q}_i(a\oplus b)={\sf q}_i(I_i(a)\oplus I_i(b))\leq I_i(a)\oplus I_i(b)\in F$ by properties of 
filters, and the interior operator. 

Now it remains to show that $F_{{\cong}_F}=F$ and $\cong_{F_{\cong}}=\cong$. 
We prove only the former. 
Let $a\in F$. Then $a=a\oplus 1$, that is $a\cong_F 1$, and so $a\in F_{{\cong}_F.}$ 
Conversely, if $a\in F_{{\cong}_F}$, then $a\cong_F 1$, that is $a\oplus 1\in F$, hence
$a\in F$ and we are done.

Finally, if $R$ is maximal, and $\Theta(R)=F$ is not a maximal filter, then there is a proper filter $J$ extending $F$ properly.
Let $x\in J\sim F$. Then $(x,1)\notin \Theta^{-1}F=R$ and $(x, 1)\in R_J$, 
so that $R$ is properly contained in the proper congruence $R_J$
which  is impossible.
\end{proof}

If $\A, \B\in \sf TCA_{\alpha}$ and $F$ is a filter of $\A$ then $\A/F$ denotes the quotient algebra $\A/\cong_F$
which is a homomorphic image of $\A$. 
If $h:\A\to \B$ is a homomorphism, then $ker h$ is the filter $\{a\in \A: h(a)=1\}$; we have
$\A/ker h\cong h(\A)$.

The next theorem summarizes some properties of filters that will be used in what follows 
without further notice. The proofs are immediate; they follow from the definitions; however we include a sketch of proof 
for the last item. For $X\subseteq \A$, $\Fl^{\A}X$ denotes the filter generated by $X$. 
If $\A\in \TCA_{\alpha}$ and $\Gamma\subseteq _{\omega}\alpha$, $\Gamma=\{i_0, \ldots, i_{n-1}\}$ say, then 
${\sf q}_{(\Gamma)}x= {\sf q}_{i_0},\ldots, {\sf q}_{i_{n-1}}x.$ This does not depend on the order of the $i_j$'s because
the ${\sf q}_i$'s $(i<\alpha)$ commute, 
and so is well defined.

\begin{lemma}  Let $\A, \B\in \TCA_{\alpha}$ with  $\B\subseteq \A$. Let $X\subseteq \A$ and $F$ be a filter of $\B$. We then have:
\begin{enumarab}
\item $\Fl^{\A}X=\{a\in A: \exists n\in \omega, x_0,\ldots, x_n\in X, \text { and }\Gamma\subseteq_{\omega} \alpha, 
{\sf q}_{(\Gamma)}(x_0\cdot x_1\cdot x_n)\leq a\}$.
\item ${\Fl^{\A}M}=\{x\in A: x\geq b \text { for some $b\in M$}\},$
\item $M={\Fl^{\A}M}\cap \B,$
\item If $\C\subseteq \A$ and $N$ is a filter of $\C$, then
${\Fl^{\A}}(M\cup N)=\{x\in A: b\cdot c\leq x\ \text { for some $b\in M$ and $c\in N$}\},$
\item For every filter $N$ of $\A$ such that $N\cap B\subseteq M$, there is a filter 
$N'$ in $\A$ such that $N\subseteq N'$ and $N'\cap B=M$.
Furthermore, if $M$ is a maximal filter of $\B$, then $N'$ 
can be taken to be a maximal filter of $\A$.
\end{enumarab}
\end{lemma}
\begin{proof} Only (iv) might deserve attention. The special case when $N=\{1\}$ is straightforward.
The general case follows from this one, by considering
$\A/N$, $\B/(N\cap \B)$ and $M/(N\cap \B)$, in place of $\A$, $\B$ and $M$ respectively.
\end{proof}

The next theorem investigates the relationship of $TIP$ and $AP$.
\begin{theorem}\label{tip} Let ${\sf K}\subseteq \sf TCA_{\alpha}$. If $\sf K$ has $AP$ and $\bold H\sf K=\sf K$ then $\sf K$ has $TIP$.
If $\bold P\sf K=\K$ and $\sf K$ has $TIP$ then it has $AP$. In particular, if $\sf K$ is a variety, then $TIP$ is equivalent to
$AP$.
\end{theorem}
\begin{proof} Assume $\sf K$ has $TIP$. Then for all $\A, \B, \C\in \sf K$, $i:\A\to \B$, $j:\A\to \C$
and $x\neq y$ in $\B$ (respectively, $x\neq y\in C$), 
there exist $\D_{xy}\in \sf K$ and homomorphisms $h_{xy}: \B\to \D_{xy}$ and $k_{xy}:\C\to \D_{xy}$ 
such that $h_{xy}\circ i=k_{xy}\circ j$ and $h_{xy}(x)\neq h_{xy}(y)$ (respectively, $k_{xy}(x)\neq k_{xy}(y)$). 
Let $\D$ be the direct product of all algebras $\D_{xy}$ for all two element sets
${x,y}$. By co-universality of $\D$ the homomorphisms
$h_{xy}:\B\to \D_{xy}$ and $k_{xy}:\C\to \D_{xy}$ induce injective 
homomorphisms $h:\B\to \D$  and $k:\C\to \D$, as required. 

Let $\A, \B, \C\in \sf K$, 
$m:\C\to \A$ be an embedding  $n:\C\to \B$ be a homomorphism. Then $\C/Ker n\cong n(\C)$.
Let $M$ be a filter of $\A$ such that $M\cap \C=ker n$. Since $\bold H\sf K=\sf K$, then $\B/M$ is in $\sf K$.
Now $\C/ker n$ embeds into $\B/M$, but it also embeds into $\C$.
By $AP$ in $\sf K$ there is a $\D\in  \sf K$ and $f: \B/M\to \D$ and $g:\C\to \D$ such that $f\circ n=g\circ m$.
Then $f^*: \B\to \D$ defined via $f^*(\bar{b})=f(b)$ 
and $g$ is as required; that is $f^*\circ n=g\circ m$,  $f^*$ is a homomorphism and $g$ is an 
embedding.
\end{proof}

Then the following can be proved. The references  \cite{P, MStwo, Sayedneat, univl, conference, es} would help a lot.
For undefined notions the reader is referred to \cite{MS, P}.
\begin{theorem}\label{apt}
Let $\alpha$ be an infinite ordinal.
\begin{enumarab}
\item $\sf TDc_{\alpha}$ has the super amalgamation property with respect to 
$\sf RTCA_{\alpha}$. In other words, 
$\sf TDc_{\alpha}$ is contained in the super amalgamation base of 
$\sf RTCA_{\alpha}$. However, it does not have even $AP.$

\item The classes of semisimple algebras  and diagonal algebras 
(as defined in \cite{P}) of dimension $\alpha$
have the amalgamation property but not
the strong amalgamation property, {\it a fortiori} they fail the super amalgamation property.

\item The free algebra $\A$ on any number of free generators $\beta>1$ 
of any variety between ${\sf TeCA}_{\alpha}$ and $\sf RTCA_{\alpha}$
has a weak form of interpolation, namely, if $X_1, X_2\subseteq \A$ and $a, b\in \Sg^{\A}X_1$ with $a\leq b$,
then there existc $c\in \Sg^{\A}{(X_1\cap X_2)}$, and a finite $\Gamma\subseteq \alpha$ such that 
${\sf q}_{(\Gamma)}a\leq c\leq {\sf c}_{(\Gamma)}b$, but it does not have
the usual interpolation property when the number of free generators are $\geq 4.$

\item Furthermore, the existence condition on the finite set $\Gamma$ in the previous item 
which makes it easier to find an interpolation cannot be 
omitted in a very strong sense. For every  finite $n\geq 0$, there is an inequality $a\leq b$ such that
the interpolant can be found using more than $n$ quantified indices of $\alpha+\omega\sim \alpha$. 
In partticular for such $n$, and such inequality the   
$\Gamma$ that provides an interpolant 
has to satisfy that $|\Gamma|>n$.

\item The former result of weak interpolation 
is equivalent
to the fact that the class of simple algebras (which is a proper class of the class of representable algebras)
have the amalgamation property.

\item The free representable algebras 
on any number of generators 
have the strong restricted interpolation property, but the free $\CA_{\alpha}$ on $\omega$
free generators does not have 
the weak interpolation property.

\item The variety $\sf RTCA_{\alpha}$ has the strong embedding property, but
the variety $\TCA_{\omega}$ does not have the embedding property.
The (strong) embedding property is a restricted form 
of the (strong) amalgamation property, namely, when the base common subagebra 
is required to be minimal.

\item $\A$ is in the amalgamation base of $\sf RTCA_{\alpha}$ iff it has the $UNEP$,
and it is in the super amalgamation base iff it has both $NS$ and 
$SUPAP$. In particular, there are representable algebras that do not have $UNEP$.

\item  $\A\in \sf RTCA_{\alpha}$ 
has $UNEP$ iff $\A$ has universal maps with respect to the neat reduct functor, in particular $\Nr$ 
does not have a right adjoint, hence it is not invertible.

\item In any class between simple algebras and representable algebras of dimension
$\alpha$ $ES$ fails, where $ES$ abbreviates that epimorphisms in the categorical
are surjective.
\end{enumarab}
\end{theorem}

Then the following can be proved. The references  \cite{P, MStwo, Sayedneat, univl, conference, es} would help a lot.
For undefined notions the reader is referred to \cite{MS, P}.
\begin{theorem}\label{apt}
Let $\alpha$ be an infinite ordinal.
\begin{enumarab}
\item $\sf TDc_{\alpha}$ has the super amalgamation property with respect to 
$\sf RTCA_{\alpha}$. In other words, 
$\sf TDc_{\alpha}$ is contained in the super amalgamation base of 
$\sf RTCA_{\alpha}$. However, it does not have even $AP.$

\item The classes of semisimple algebras  and diagonal algebras 
(as defined in \cite{P}) of dimension $\alpha$
have the amalgamation property but not
the strong amalgamation property, {\it a fortiori} they fail the super amalgamation property.

\item The free algebra $\A$ on any number of free generators $\beta>1$ 
of any variety between ${\sf TeCA}_{\alpha}$ and $\sf RTCA_{\alpha}$
has a weak form of interpolation, namely, if $X_1, X_2\subseteq \A$ and $a, b\in \Sg^{\A}X_1$ with $a\leq b$,
then there existc $c\in \Sg^{\A}{(X_1\cap X_2)}$, and a finite $\Gamma\subseteq \alpha$ such that 
${\sf q}_{(\Gamma)}a\leq c\leq {\sf c}_{(\Gamma)}b$, but it does not have
the usual interpolation property when the number of free generators are $\geq 4.$

\item Furthermore, the existence condition on the finite set $\Gamma$ in the previous item 
which makes it easier to find an interpolation cannot be 
omitted in a very strong sense. For every  finite $n\geq 0$, there is an inequality $a\leq b$ such that
the interpolant can be found using more than $n$ quantified indices of $\alpha+\omega\sim \alpha$. 
In partticular for such $n$, and such inequality the   
$\Gamma$ that provides an interpolant 
has to satisfy that $|\Gamma|>n$.

\item The former result of weak interpolation 
is equivalent
to the fact that the class of simple algebras (which is a proper class of the class of representable algebras)
have the amalgamation property.

\item The free representable algebras 
on any number of generators 
have the strong restricted interpolation property, but the free $\CA_{\alpha}$ on $\omega$
free generators does not have 
the weak interpolation property.

\item The variety $\sf RTCA_{\alpha}$ has the strong embedding property, but
the variety $\TCA_{\omega}$ does not have the embedding property.
The (strong) embedding property is a restricted form 
of the (strong) amalgamation property, namely, when the base common subagebra 
is required to be minimal.

\item $\A$ is in the amalgamation base of $\sf RTCA_{\alpha}$ iff it has the $UNEP$,
and it is in the super amalgamation base iff it has both $NS$ and 
$SUPAP$. In particular, there are representable algebras that do not have $UNEP$.

\item  $\A\in \sf RTCA_{\alpha}$ 
has $UNEP$ iff $\A$ has universal maps with respect to the neat reduct functor, in particular $\Nr$ 
does not have a right adjoint, hence it is not invertible.

\item In any class between simple algebras and representable algebras of dimension
$\alpha$ $ES$ fails, where $ES$ abbreviates that epimorphisms in the categorical
are surjective.
\end{enumarab}
\end{theorem}

We recall the definitions from \cite{Sayedneat}:

\begin{definition}\cite[def. 5.2.1]{Sayedneat} Let $\A\in \sf RTCA_{\alpha}$. Then $\A$ has the $UNEP$ (short for unique neat embedding property)
if for all $\A'\in \TCA_{\alpha}$, $\B$, $\B'\in \TCA_{\alpha+\omega},$
isomorphism $i:\A\to \A'$, embeddings  $e_A:\A\to \Nr_{\alpha}\B$
and $e_{A'}:\A'\to \Nr_{\alpha}\B'$ such that $\Sg^{\B}e_A(A)=\B$ and $\Sg^{\B'}e_{A'}(A)'=\B'$, there exists
an isomorphism $\bar{i}:\B\to \B'$ such that $\bar{i}\circ e_A=e_{A'}\circ i$.
\end{definition}

\begin{definition}\cite[def. 5.2.2]{Sayedneat} Let $\A\in \sf RTCA_{\alpha}$. 
Then $\A$ has the $NS$ property (short for neat reducts commuting with forming subalgebras)
if for all $\B\in \TCA_{\alpha+\omega}$ if $\A\subseteq \Nr_{\alpha}\B$ then for all
$X\subseteq A,$ $\Sg^{\A}X=\Nr_{\alpha}\Sg^{\B}X$.
\end{definition}

\begin{lemma} If $\alpha\geq \omega$ and $\A\in \TDc_{\alpha}$ then $\A$ has $NS$ and $UNEP.$
\end{lemma}
\begin{proof}
\cite[Theorem 2.6.67-71-72]{HMT1}.
\end{proof}
\begin{theorem}\label{SUPAP} Let $\alpha$ be an infinite ordinal.
\begin{enumarab}
\item $\sf TDc_{\alpha}\subseteq {\sf SUPAPbase(\sf RTCA_{\alpha})}$.
\item $\sf TLf_{\alpha}$ has $SUPAP$.
\end{enumarab}
\end{theorem}
\begin{proof}
For the first part. Let $\beta=\alpha+\omega$.
Let $\C\in \Dc_{\alpha}$, let $\A, \B\in \sf RTCA_{\alpha}$, and let $f:\C\to \A$ and $g:\C\to \B$ be injective homomorphisms.
Then there exist $\A^+, \B^+, \C^+\in \TCA_{\beta}$, $e_A:\A\to \Nr_{\alpha}\A^+$
$e_B:\B\to  \Nr_{\alpha}\B^+$ and $e_C:\C\to \Nr_{\alpha}\C^+$.
We can assume, without loss,  that $\Sg^{\A^+}e_A(A)=\A^+$ and similarly for $\B^+$ and $\C^+$.
Let $f(C)^+=\Sg^{\A^+}e_A(f(C))$ and $g(C)^+=\Sg^{\B^+}e_B(g(C)),$ so that $\A^+, \B^+$ and $\C^+$ are in $\bold K$.
Then by the $UNEP$, there exist $\bar{f}:\C^+\to f(C)^+$ and $\bar{g}:\C^+\to g(C)^+$ such that
$(e_A\upharpoonright f(C))\circ f=\bar{f}\circ e_C$
and $(e_B\upharpoonright g(C))\circ g=\bar{g}\circ e_C$. Both $\bar{f}$ and $\bar{g}$ are injective homomorphisms
and $\bold K$ has $SUPAP$, hence there is a $\D^+$ in $\bold K$ and $k:\A^+\to \D^+$ and $h:\B^+\to \D^+$ such that
$k\circ \bar{f}=h\circ \bar{g}$. Also $k$ and $h$ are injective homomorphisms. Then $k\circ e_A:\A\to \Nr_{\alpha}\D^+$ and
$h\circ e_B:\B\to \Nr_{\alpha}\D^+$ are one to one and
$k\circ e_A \circ f=h\circ e_B\circ g$.
Let $\D=\Nr_{\alpha}\D^+$. Then we have obtained $\D\in \Nr_{\alpha}\TCA_{\alpha+\omega}$
and $m:\A\to \D$, $n:\B\to \D$
such that $m\circ f=n\circ g$.
Here $m=k\circ e_A$ and $n=h\circ e_B$.
We have proved $AP$.

Denote $k$ by $m^+$ and $h$ by $n^+$.
We further want to show that if $m(a) \leq n(b)$,
for $a\in A$ and $b\in B$, then there exists $t \in \C$
such that $ a \leq f(t)$ and $g(t) \leq b$.
So let $a$ and $b$ be as indicated.
We have  $(m^+ \circ e_A)(a) \leq (n^+ \circ e_B)(b),$ so
$m^+ ( e_A(a)) \leq n^+ ( e_B(b)).$
Since $\bold K$ has $SUPAP$, there exist $z \in C^+$ such that $[e_A(a)] \leq \bar{f}(z)$ and
$\bar{g}(z) \leq [e_B(b)]$.
Let $\Gamma = \Delta z \sim \alpha$ and $z' =
{\sf c}_{(\Gamma)}z$. (Note that $\Gamma$ is finite.) So, we obtain that
$e_A({\sf c}_{(\Gamma)}a) \leq \bar{f}({\sf c}_{(\Gamma)}z)~~ \textrm{and} ~~ \bar{g}({\sf c}_{(\Gamma)}z) \leq
e_B({\sf c}_{(\Gamma)}b).$ It follows that $e_A(a) \leq \bar{f}(z')~~\textrm{and} ~~ \bar{g}(z') \leq e_B(b).$ Now by the $NS$ property for $\sf TDc$, 
we have
$z' \in \Nr_\alpha \C^+ = \Sg^{\Nr_\alpha \C^+} (e_C(C)) = e_C(C).$
So, there exists $t \in C$ with $ z' = e_C(t)$. Then we get
$e_A(a) \leq \bar{f}(e_C(t))$ and $\bar{g}(e_C(t)) \leq e_B(b).$ It follows that $e_A(a) \leq (e_A \circ f)(t)$ and
$(e_B \circ g)(t) \leq
e_B(b).$ Hence, $ a \leq f(t)$ and $g(t) \leq b.$
We are done.
For the locally finite case, one takes the subalgebra of $\Nr_{\alpha}\D$ (as constructed above)
generated by the images of $\A$ and $\B$ which are now locally finite, as an amalgam,
and hence as a super amalgam. This algebra is necessarily locally finite.
\end{proof}

In the above theorem, we do not guarantee that the super amalgam is found inside $\sf TDc_{\alpha}$, for the subalgebra of the
amalgam as formed for the locally finite case may not be in $\sf TDc_{\alpha}$.
Indeed, we have:

\begin{theorem}\label{d}For $\alpha\geq \omega$, $\sf TDc_{\alpha}$ does not have $AP.$
\end{theorem}
\begin{proof} Let $\A, \B\in \sf TDc_{\alpha}$, such that their minimal subalgebras 
are isomorphic and for which there exist $x\in \A$ and $y\in \B$, such that $\Delta x\cup \Delta y=\alpha$.
Clearly such algebras cannot be amalgamated by a $\Dc_{\alpha}$ over the common  
minimal subalgebra $\M$ say of $\A$ and $\B$, embedded into each by the inclusion map $i$. 
For if $\C\in \sf TDc_{\alpha}$ and $m:\A\to \C$ and $n:\B\to \C$, such that $m\circ i=n\circ i$, 
then $\Delta(m(x)+n(y))=\alpha$ which is not possible because the amalgam is 
assumed to be dimension complemented.
\end{proof}

Using the same argument as in theorem \ref{tip} it can be shown:
\begin{theorem}
If $\C\in \sf TDc_{\alpha}$, $\A, \B\in \sf TRCA_{\alpha}$
$n:\C\to \A$ an embedding and $m:\C\to \B$ a homomorphism, then there exists $\D\in \sf TRCA_{\alpha}$
a homomorphism $f:\A\to \D$ and an embedding  $g:\B\to \D$  
such that $f\circ n=g\circ m$. 
\end{theorem}

Now we deal with concepts that are localizations of the amalgamation property; in the sense
that various  amalgamation properties will be proved to hold for a class of algebras if the free algebras 
of such classes enjoy such local properties, typically {\it interpolation} properties.
We have already dealt with one such property; we introduce weaker ones. 

\begin{definition} Let $\alpha$ be any ordinal (finite included) and $\A\in \TCA_{\alpha}$. Then
\begin{enumarab}
\item $\A$ has the {\it weak interpolation property}, 
$WIP$ for short, if for all $X_1, X_1\subseteq \A$, for all $x\in \Sg^{\A}X_1$, $z\in \Sg^{\A}X_2$ if
$x\leq z$, then 
there exist $\Gamma\subseteq_{\omega}\alpha$ and $y\in \Sg^{\A}(X_1\cap X_2)$ such that 
$${\sf q}_{(\Gamma)}x\leq y\leq {\sf c}_{(\Gamma)}z.$$

\item $\A$ has the {\it universal interpolation property}, $UIP$ for short, if for all $X_1, X_1\subseteq \A$, for all $x\in \Sg^{\A}X_1$, $z\in \Sg^{\A}X_2$ if
$x\leq z$, then 
there exist $\Gamma\subseteq_{\omega}\alpha$ and $y\in \Sg^{\A}(X_1\cap X_2)$ such that 
$${\sf q}_{(\Gamma)}x\leq y\leq z.$$

\item $\A$ has the {\it existential interpolation property}, $EIP$ 
for short, if if for all $X_1, X_1\subseteq \A$, for all $x\in \Sg^{\A}X_1$, $z\in \Sg^{\A}X_2$ if
$x\leq z$, then 
there exist $\Gamma\subseteq_{\omega}\alpha$ and $y\in \Sg^{\A}(X_1\cap X_2)$ such that 
$$x\leq y\leq {\sf c}_{(\Gamma)}z.$$
\end{enumarab}

\end{definition}

The following theorem is proved by a {\it compactness argument} taken from \cite{typeless}. The proof works only for infinite dimension. 
We shall see that the theorem fails when the dimension is finite. 
It can be used in even a much wider context, saying that 
if the {\it dimension restricted free} algebras in $\omega$ extra dimensions have the interpolation property, 
then the free algebras {\it without any restrictions}
have a natural weak form of interpolation. The idea is that an interpolant can always be found if we allow 
infinitely many more dimensions (variables), though only finitely many are used
in the interpolant. Then quantifiers may be  used to get rid of the extra variables bouncing the interpolant back to using only
the number of available variables. For a term $\sigma$ in the language of $\sf TCA_{\alpha},$ ${\sf Var}(\sigma)$ denotes
the set of variables occuring in $\sigma$.

\begin{theorem}\label{typeless}
Let $\alpha\geq \omega$. Let $\sf K$ be a class of algebras such that ${\sf RTCA}_{\alpha}\subseteq {\sf K}\subseteq {\sf TCA}_{\alpha}.$
Then for any terms of the language of $\TCA_{\alpha}$, $\sigma, \tau$ say,  if
${\sf K}\models \sigma\leq \tau$, then there exist  
a term $\pi$ with ${\sf Var}(\pi)\subseteq {\sf Var}(\sigma)\cap {\sf Var}(\tau)$ and a finite $\Delta\subseteq \alpha$ such that
$$\sf K\models {\sf q}_{(\Delta)}\sigma\leq \pi\leq {\sf c}_{(\Delta)}\tau.$$ In particular, for any non-zero cardinal $\beta,$
$\Fr_{\beta}{\sf K}$, has the $WIP$.
\end{theorem}
\begin{proof} The same argument in \cite{typeless} but now using theorem \ref{in}.
\end{proof}
We will show that the above form of interpolation is the best possible for such classes $\sf K$, witness theorem \ref{infinitevarieties}.
Furthermore it fails for finite dimensions, theorem \ref{notequivalent}.

We have proved that for $\alpha\geq \omega$, $\sf TDc_{\alpha}$ lies in the super amalgamation base of $\sf TRCA_{\alpha}$.
In what follows we define larger classes of algebras, still retaining an amalgamation property.
We lose the {\it strong} amalgamation property, but in return in such cases the amalgam is always found 
inside the class in question. 

The following class was introduced by Monk for cylindric algebras 
under the name of {\it Diagonal cylindric algebras}, and was denoted by ${\sf Di}_{\alpha}$ in 
\cite{P}. In the last reference Pigozzi proved that this class has $AP$. 
We obtain an anlogous result, but we define the class differently, 
allowing generalization to diagonal free
algebras. We use the term definable substitutions corresponding to replacements. 
In what follows $\alpha$, unless indicated otherwise,  is infinite.
\begin{definition}
$\A\in \TCA_{\alpha}$
is called a {\it substitution algebra of dimension $\alpha$}, if for all non-zero $x$ in $A$,
for all finite $\Gamma\subseteq \alpha$, there exist distinct
$i,j\in \alpha\sim \Gamma$, such that ${\sf s}_i^jx\neq 0$.
Let $\sf TSc_{\alpha}$ denote the class of
substitution algebras.
\end{definition}

%Now we prove that such class consists only of representable $\TCA_{\alpha}$s, using the neat embedding 
%theorem \ref{neat}.

We know that any simple algebra is in $\Sc_{\alpha}$ and semi-simple algebras are subdirect products of
simple ones. This does not guarantee that the class of semi-simple algebras is contained in $\Sc_{\alpha}$
because we cannot assume a priori that the latter class  is closed under products.
However, as it happens, we have:

\begin{theorem} If $\A$ is semi-simple or $\A\in \TDc_{\alpha}$, then $\A\in \Sc_{\alpha}.$
\end{theorem}
\begin{proof} Let $\A\in \TDc_{\alpha}$ $a\in \A$ be non zero, and $\Gamma\subseteq_{\omega} \alpha$.
Choose $i,j\in \alpha\sim \Delta x$. Then ${\sf s}_{i}^jx=x\neq 0$.
Let $\A\in \Sc_{\alpha}$, $\Gamma$ be a finite subset of $\alpha$ and $x\in \A\sim \{0\}$.
Using Zorn's lemma one can find  a maximal filter $F$ of $\A$ such that $x\notin F$.
Since $F$ is maximal then $\A/F$ is simple. But $x\notin F$, hence  there exists a finite $\Delta\subseteq \alpha$,
such that ${\sf c}_{(\Delta)}(x/F)={\sf c}_{(\Delta)} x/F=1$.

Let $i,j\in \alpha\sim (\Gamma\cup \Delta)$, then we claim that
${\sf s}_i^jx\neq 0$. If not, then
$$0=({\sf c}_{(\Delta)}{\sf s}_i^jx)/F={(\sf s}_i^j{\sf c}_{(\Delta)}x)/F =
{\sf c}_{(\Delta)}x/F= 1,$$  which is impossible.
\end{proof}
\begin{theorem}\label{sc} $\Sc_{\alpha}\subseteq \sf RTCA_{\alpha}$. Furthermore, 
the class $\Sc_{\alpha}$ has $AP$  and $TIP$.
If the algebras to be amalgamated are semi-simple 
then the amalgam can be chosen to be semisimple, too. The same holds when we replace
semi-simple by simple.
\end{theorem}
\begin{proof} The same argument in \cite[Theorem 2.2.24]{P} using theorem \ref{s}.
\end{proof}
\begin{theorem}\label{ultra} ${\bf SUp}\sf TSc_{\alpha}=\sf RTCA_{\alpha}.$
\end{theorem}
\begin{proof}
From theorem \ref{sc}  since $\sf TDc_{\alpha}\subseteq \Sc_{\alpha}.$
\end{proof}

The following class is the $\TCA$ analogue of the class of cylindric algebras introduced 
in item (iii) of theorem 2.6.50. Using the neat embedding theorem \cite[Theorem 4.2(2)]{part1} yet again, together with ultraproducts, 
we show that such a class consists only
of representable $\TCA_{\alpha}$s, furthermore, it is easy to see that $\sf TSc_{\alpha}$ 
is contained in it, and as we shall see 
in a minute properly.

\begin{definition} $\A\in \TCA_{\alpha}$ is called a {\it weak substitution algebra},
a $\sf TWSc_{\alpha}$ for short, if for every finite injective map $\rho$ into $\alpha$,
and for every  $x\in A$, $x\neq 0$, there is a function $h$ and $k<\alpha$ such
that $h$ is an endomorphism of
$\Rd^{\rho}\A$, $k\in \alpha\sim \rng(\rho)$, ${\sf c}_k\circ h=h$ and $h(x)\neq 0$.
\end{definition}
The following theorem can be proved exactly like in \cite[Theorem 2.6. 50(iii)]{HMT1} using the hitherto established 
neat embedding theorem \cite[Theorem 4.2 (2)]{part1}.

\begin{theorem}\label{wsc}
For any infinite ordinal $\alpha$, $\sf TWSc_{\alpha}\subseteq \sf TRCA_{\alpha}$.
\end{theorem}

\begin{theorem}\label{twsc}

The following conditions are equivalent:

\begin{enumarab}
\item $\sf TWSc_{\alpha}$ is elementary.
\item $\sf TWSc_{\alpha}$ is closed under ultraproducts.
\item $\sf TWSc_{\alpha}=\sf TRCA_{\alpha}.$
\item $\sf TWSc_{\alpha}$ is a variety.
\end{enumarab}
\end{theorem}
\begin{proof}
\begin{enumroman}
\item $(1)\to (2)$ is trivial.

\item $(2)\to (3)$
If $\sf TWSc_{\alpha}$ is closed under ultraproducts, then because it is closed under forming subalgebras, we have
$\sf TWSc_{\alpha}={\bf SUp}\sf TWSc_{\alpha}=
\sf TRCA_{\alpha}$.

\item $(3)\to (4)$ Since $\sf RTCA_{\alpha}$ is a variety.

\item $(4)\to (1)$ Trivial.

\end{enumroman}
\end{proof}

We let $\sf TSs_{\alpha}$ stand for semisimple algebras. 
In the next example we show that the inclusions 
$$\sf TDc_{\alpha}\subseteq \sf TSs_{\alpha}\subseteq \sf TSc_{\alpha}\subseteq \sf TWSs_{\alpha}\subseteq \sf RTCA_{\alpha}$$ 
are all proper for $\alpha\geq \omega$, for the $\sf CA$ case witness \cite[Remark 2.6.51]{HMT1}.

But first we need a definition.

\begin{definition} 
Let $\alpha$ be any ordinal $>1$ and  $\A\in \sf TCA_{\alpha}$. Then $I$ is an ideal in $\A$ iff $I$ is an ideal in $\Rd_{ca}\A$.
\end{definition}
It can be easily checked that this definition is sound in the sense that ideals so 
defined correpond to congruences (hence to filters) via
$\cong\mapsto \{a\in A: a\cong 0\}$.

\begin{example}\label{strict}

\begin{enumarab}
\item For the first inclusion. Let $m\geq 2$ be a finite ordinal. Take $\A=\wp(^{\alpha}m)$,  it is easy to
see that $\A\in \sf TSc_{\alpha}$. However,
$\A$ is not in $\sf TDc_{\alpha}$ because for every
$s\in {}^{\alpha}m$, we have
$\Delta( \{s\})=\alpha$. 
We show that $\A$  is not even semi-simple by showing that for any constant map $f:\alpha\to m$, the singleton $\{f\}$  is in all the maximal
proper ideals. Let $f$ be such a map.
Let $X=\{f\}$. 
Let $J$ be a maximal proper ideal.
Assume for contradiction that $X\notin J$. Then $X/J\neq 0$ in $\A/J$. Since $\A/J$ is simple, then 
the ideal generated by $X/J$ coincides with $\A/J$, so there exists a finite $\Gamma\subseteq \alpha$
such that ${\sf c}_{(\Gamma)}(X/J)={\sf c}_{(\Gamma)}X/J= 1$. 
This means that ${\sf c}_{(\Gamma)}X\notin J$, but $J$ is maximal, hence 
$- {\sf c}_{(\Gamma)}X\in J$.
Now let $k\in \alpha\sim \Gamma$, and let $t$ be the sequence that agrees with $f$
everywhere except at $k$, where its value is $\neq f(k)$.
Then $t\in -{\sf c}_{(\Gamma)}X$, so $\{t\}\in J$ by maximality of $J$.
But $X\subseteq {\sf c}_k\{t\}$, so $X\in J$ which is
impossible.

\item  let $\A=\wp(^{\alpha}\alpha)$; then of course $\A\in \sf TRCA_{\alpha}$.
We show that $\A\notin \sf TSc_{\alpha}$.
Let $\Theta$ be a bijection from $\alpha$ to $\alpha$
and consider the element $x=\{\Theta\}\in \A$.

Then for any distinct $i,j\in \alpha$ ${\sf s}_i^jx=0$, because $\sigma\in {\sf s}_i^jx$ 
iff $\sigma\circ [i|j]=\Theta$ which is  impossible, because 
$\sigma\circ [i|j](i)=\sigma\circ[i|j](j)$.

We now show that $\A\in \sf TWSc_{\alpha}$.
Let $x\in \A$ be non-zero.
Let   $\rho$ be a one to one finite function with $\rng(\rho)\subseteq \alpha$.
We want to find $H$ as in the conclusion of the definition of a $\sf TWSc_{\alpha}$.
Let $\tau\in {}^{\alpha}\alpha$ such that $k\notin \rng\tau$,
$\tau\upharpoonright \rng \rho\subseteq Id$ and $\tau$ is one to one.
Let $H: \A\to \A$ by
$H(Y)=\{\phi \in {}\wp(^{\alpha}\alpha): \phi\circ \tau \in Y\}.$
Then $H$ is as required.

\end{enumarab}
\end{example}

\begin{corollary}\label{el} Any class $\sf K$, such that $\sf TLf_{\alpha}\subseteq{\sf K}\subseteq \sf TSs_{\alpha}$ is not closed under
ultraproducts, hence is not elementary. The class of semisimple algebras is not closed under $\bf H$. 
\end{corollary}
\begin{proof}
Let $\A$ be a simple, locally finite, non-discrete cylindric algebra of dimension $\alpha$,
Here non-discrete means that there is an $i\in \alpha$ such that
${\sf c}_i\neq Id$ expanded by $I_i=Id$ for all $i<\alpha$.

Let $I$ be an infinite set and $J=\{\Gamma: \Gamma\subseteq I, |\Gamma|<\omega\}$.
For $\Gamma\in J$, let $M_{\Gamma}=\{\Delta\in J: \Gamma\subseteq \Delta\}$.
Let $F$ be an ultrafilter on $J$ that contains $M_{\Gamma}$ for every $\Gamma\in J$; clearly exists
for $M_{\Gamma_1}\cap M_{\Gamma_2}=M_{\Gamma_1\cup \Gamma_2}.$

Let $\B$ be the following ultrapower of $\A$, $\B={}^J\A/F$. Then it is proved in
\cite[remark 2.4.59]{HMT1} that $\Rd_{ca}\B$ is not semi-simple and not dimension complemented, 
hence $\B$ is not semi-simple, because
$I$ is an ideal in $\B$ if and only if it is an ideal in $\Rd_{ca}\B$. 
Obviously $\B$ is also not dimension complemented and we are done.

For the last part since ${\bf SP}\Ss_{\alpha}=\Ss_{\alpha}$ and the latter is not a variety, hence
${\bold H}\Ss_{\alpha}\neq \Ss_{\alpha}$.

\end{proof}

\begin{corollary} \label{ultraproducts}$\sf TSc_{\alpha}$ is not closed under ultraproducts hence it is not first order axiomatizable, least a variety.
\end{corollary}
\begin{proof} $\sf TSc_{\alpha}$ is clearly closed under forming subalgebras,
hence by theorem \ref{ultra} and example \ref{strict}, it is not closed under
ultraproducts.
\end{proof}

\section{Interpolation and amalgamation}

If we do not have an order (as we do now) what corresponds locally to amalgamation properties, are
{\it congruence extension properties}.
So let us see how these relate to the interpolation properties defined earlier.
Recall that for an algebra $\A$, 
$\Co\A$ stands for the set of all congruence relations on $\A$.

\begin{definition} Let $\A$ be an algebra, and $C\subseteq \bigcup_{\B\subseteq \A}\Co\A$.
$\A$ is said to have {\it the congruence extension property relative to
$C$} if for any $X_1, X_1\subseteq \A$ such that $X_1\cap X_2=X$,
if $R\in \Co(\Sg^{\A}X_1)\cap C$ and $S\in \Co(\Sg^{\A}X_2)\cap C$,
such that $R\cap {}^2\Sg^{\A}(X_1\cap X_2)=S\cap {}^2\Sg^{\A}(X_1\cap X_2),$
then there is a $T\in (\Co\A)\cap C$
such that $T\cap {}^2\Sg^{\A}X_1=R$ and $T\cap {}^2\Sg^{\A}X_2=S$.
If $C=\bigcup_{\B\subseteq \A}\Co\A$, we say that $\A$ 
has {\it the congruence extension property}, or $CP$ for short.
\end{definition}

From now on $\alpha$ is an arbitrary ordinal $>0$. 
We next formulate the above property for free algebras in various forms. We will see that properties of the free algebras of a variety may be reflected
in properties of the corresponding equational consequence relations of the variety, in particular we may focus on properties 
of the equational consequence relation for a {\it countable} list of variables,
and this enables us to restrict our attention to countable free algebras only as shown by 
G. Metcalfe et al. \cite{george}.
The {\it Pigozzi property $PP$}, the {\it Robinson property $RP$}, the {\it Maehara interpolation property $MIP$}, {\it 
the deductive interpolation property $DIP$}  are defined in \cite{george}.
{\it Countable $MIP$}, {\it  countable $DIP$} and {\it countable $RP$}, are the restriction of such properties 
when the variables available are countable.

We note that $MIP$ is the interpolation property corresponding to $TIP$
which in turn is equivalent in varieties in which congruences of subalgebras of an algebra lift to congruences of the algebra 
(which is our case)\footnote{This property is referred to in the literature as the 
{\it congruence extension property}, but we do not use this term here for we have reserved
the term congruence extension property for a different property.} to  $AP$ \cite{george}.
\begin{theorem}\label{app} Let $\alpha$ be an ordinal $>0$. 
Let $V$ be a subvariety of $\sf TCA_{\alpha}$. Then 
the following conditions are equivalent:
\begin{enumarab}
\item $V$ has the $TIP$.
\item $V$ has the (countable) $MIP.$
\item $V$ has the (countable) $DIP$.
\item $V$ has $AP.$
\item Finitely generated algebras in $V$ has $AP$
\item The free algebras have the $CP.$
\item The (countable) free algebras have the $UIP$.
\item The (countable) free algebras have $EIP.$
\item $V$ has $PP.$
\item $V$ has the (countable) $RP.$
\end{enumarab}
\end{theorem} 
\begin{proof}
It is known that for any variety $V$ \cite{george} all of  (1)-(6) and (9)-(10) 
are equivalent to each other.

We prove $(4)\implies (6)$ in a form to be used later on.  
Assume that $V$ has $AP$ and let $\A$ be the free algebra on a non empty set of generators.
For $R\in \Co\A$ and $X\subseteq  A$, by $(\A/R)^{(X)}$ we understand the subalgebra of
$\A/R$ generated by $\{x/R: x\in X\}.$ We want to show that $\A$ has $CP$.
Let $\A$, $X_1$, $X_2$, $R$ and $S$ be as specified in in the definition of $CP$.
Define $$\theta: \Sg^{\A}(X_1\cap X_2)\to \Sg^{\A}X_1/R$$
by $$a\mapsto a/R.$$
Then $ker\theta=R\cap {}^2\Sg^{\A}(X_1\cap X_2)$ and $Im\theta=(\Sg^{\A}X_1/R)^{(X_1\cap X_2)}$.
It follows that $$\bar{\theta}:\Sg^{\A}(X_1\cap X_2)/R\cap {}^2\Sg^{\A}(X_1\cap X_2)\to (\Sg^{\A}X_1/R)^{(X_1\cap X_2)}$$
defined by
$$a/R\cap {}^{2}\Sg^{\A}(X_1\cap X_2)\mapsto a/R$$
is a well defined isomorphism.
Similarly
$$\bar{\psi}:\Sg^{\A}(X_1\cap X_2)/S\cap {}^2\Sg^{\A}(X_1\cap X_2)\to (\Sg^{\A}X_2/S)^{(X_1\cap X_2)}$$
defined by
$$a/S\cap {}^{2}\Sg^{\A}(X_1\cap X_2)\mapsto a/S$$
is also a well defined isomorphism.
But $$R\cap {}^2\Sg^{\A}(X_1\cap X_2)=S\cap {}^2\Sg^{\A}(X_1\cap X_2),$$
Hence
$$\phi: (\Sg^{\A}X_1/R)^{(X_1\cap X_2)}\to (\Sg^{\A}X_2/S)^{(X_1\cap X_2)}$$
defined by
$$a/R\mapsto a/S$$
is a well defined isomorphism.
Now
$(\Sg^{\A}X_1/R)^{(X_1\cap X_2)}$ embeds into $\Sg^{\A}X_1/R$
via the inclusion map; it also embeds in $\Sg^{\A}X_2/S$ via $i\circ \phi$ where $i$
is also the inclusion map.
For brevity let $\A_0=(\Sg^{\A}X_1/R)^{(X_1\cap X_2)}$, $\A_1=\Sg^{\A}X_1/R$ and $\A_2=\Sg^{\A}X_2/S$ and $j=i\circ \phi$.
Then $\A_0$ embeds in $\A_1$ and $\A_2$ via $i$ and $j$ respectively.
Then there exists $\B\in V$ and injective homomorphisms
$f$ and $g$ from $\A_1$ and $\A_2$ respectively to
$\B$ such that $f\circ i=g\circ j$.
Let $$\bar{f}:\Sg^{\A}X_1\to \B$$ be defined by $$a\mapsto f(a/R)$$ and $$\bar{g}:\Sg^{\A}X_2\to \B$$
be defined by $$a\mapsto g(a/R).$$
Let $\B'$ be the algebra generated by $\rng f\cup \rng g$.
Then $\bar{f}\cup \bar{g}\upharpoonright X_1\cup X_2\to \B'$ is a function since $\bar{f}$ and $\bar{g}$ coincide on $X_1\cap X_2$.
By freeness of $\A$, there exists $h:\A\to \B'$ such that $h\upharpoonright_{X_1\cup X_2}=\bar{f}\cup \bar{g}$.
Let $T=kerh $. Then it is not hard to check that
$$T\cap {}^2 \Sg^{\A}X_1=R \text { and } T\cap {}^2\Sg^{\A}X_2=S.$$
$T$ induces the required congruence.

$(6)\implies (7)$. Let $x\in \Sg^{\A}X_1$, $z\in \Sg^{\A}X_2$ and assume that $x\leq z$.
Then $$z\in (\Fl^{\A}\{x\})\cap \Sg^{\A}X_1).$$
Let $$M=\Fl^{\Sg^{\A}X_1}\{x\}\text { and } N=\Fl^{\Sg^{\A}X_2}(M\cap \Sg^{\A}(X_1\cap X_2)).$$
Then $$M\cap \Sg^{\A}(X_1\cap X_2)=N\cap \Sg^{\A}(X_1\cap X_2).$$
By identifying ideals with congruences, and using the congruence extension property,
there is a filter $P$ of $\A$
such that $$P\cap \Sg^{\A}X_1=N\text { and }P\cap \Sg^{\A}X_2=M.$$
It follows that
$$\Fl^{\A}(N\cup M)\cap \Sg^{\A}X_1\subseteq P\cap \Sg^{\A}X_1=N.$$
Hence
$$(\Fl^{(\A)}\{z\})\cap \Sg^{\A}X_1\subseteq N.$$
and we have
$$z\in \Fl^{\Sg^{\A}X_1}[\Fl^{\Sg^{\A}X_2}\{x\}\cap \Sg^{\A}(X_1\cap X_2).]$$his implies that there is an element $y$
such that $$z\geq y\in  \Sg^{\A}(X_1\cap X_2),$$
and $y\in \Fl^{\Sg^{\A}X_2}\{x\}$. Hence, there exists a finite
$\Gamma\subseteq \alpha$ such that $y\geq {\sf q}_{(\Gamma)}x$,
so we get
$${\sf q}_{(\Gamma)}x\leq y\leq z.$$

$(7)\implies (1)$. Let $\A, \B, \C\in V$, with inclusions
$m:\C\to \A$, $n:\C\to \B$. We want to find an amalgam. Let $\D$ be the free algebra on $|D|$ generators, where
$|D|>max{|B|, |C|}$ for all $i<|D|$.
Let $h:\D\to \C$,
$h_1:\D\to \A$, $h_2:\D\to \B$ be homomorphisms
such that for $x\in h^{-1}(\C)$,
$$h_1(x)=m\circ h(x)=n\circ h(x)=h_2(x).$$
Such homomorphisms clearly exist by the freeness, cardinality of $\D$, and the fact that $\A, \B, \C\in V$,
Let $\D_1=h_1^{-1}(\A)$ and $\D_2=h_2^{-1}(\B)$. Then $h_1:\D_1\to \A$, and $h_2:\D_2\to \B$.
Let $M=ker h_1$ and $N=ker h_2$, and let
$\bar{h_1}:\D_1/M\to \A, \bar{h_2}:\D_2/N\to \B$ be the induced isomorphisms.
Let $l_1:h^{-1}(\C)/h^{-1}(\C)\cap M\to \C$ be defined via $\bar{x}\to h(x)$, and
$l_2:h^{-1}(\C)/h^{-1}(\C)\cap N$ to $\C$ be defined via $\bar{x}\to h(x)$.
Then those are well defined, and hence
$k^{-1}(\C)\cap M=h^{-1}(\C)\cap N$.
We show that $\Fl^{\D}(M\cup N)$ is a proper filter  of $\D$ and that $\D/P$ is the required amalgam.
Let $x\in \Fl^{\D}(M\cup N)\cap \D_1$.
Then there exist $b\in M$ and $c\in N$ such that $b\cdot c\leq x$.
Thus $c\leq x+ -b$.
But $x+ -b\in \D_1$ and $c\in \D_2$, it follows by assumption that there exist
$d\in \D_1\cap \D_2$ such that ${\sf q}_{(\Gamma)}c\leq d\leq x+ -b$.
Notice that $c\in N$ so ${\sf q}_{(\Gamma)}c\in N$, hence $d\in N$, so 
$d\in M$,  because $M\cap \D_1\cap \D_2=N\cap \D_1\cap \D_2$.   Hence $x\in M.$
Thus $P=\Fl^{\D}(M\cup N)$ is proper, and $\D/P$ is the required amalgam.

Countable free algebras have $UIP$ implies $AP$, because we can restrict our 
attention only to countable algebras being 
amalgamated. $AP$ equivalent to countable $EIP$ is exactly like above by working 
with ideals instead of filters.
\end{proof}

Since $UIP$ and $EIP$ are equivalent in the case of varieties we call the (one) property they express {\it the almost interpolation property}, 
briefly $AIP$.

Before our next theorem which provides infinitely many varieties satisfying the conditions of theorem \ref{app}.
For $R\subseteq_{\omega} \alpha\times \alpha$, $\bar{{\sf d}}(R)=\prod_{i,j\in R\sim Id} -{\sf d}_{ij}.$
The next result is the topological analogue  of a result of Comer \cite{Comer} 
proved for cylindric algebras,  but we present a different proof depending on theorem \ref{sc}.

\begin{theorem} If $\alpha\geq \omega$, then there 
are infinitely many subvarieties of ${\sf RTCA}_{\alpha}$ that has $AP,$ hence satisfy all conditions in theorem
\ref{app}.
\end{theorem}
\begin{proof} For each finite $k \geq 1$, let $V_k$ be the variety consisting of algebras having characteristic $k$ endowed with interior operators. 
$\A\in V_k$ iff $\A\in \sf TCA_{\alpha}$ and $\A$ satisfies the equations ${\sf c}_{(k)}\bar{{\sf d}}(k\times k)=1$ and 
${\sf c}_{(k+1)}\bar{{\sf d}}((k+1)\times (k+1))=0$.
Then $V_k\subseteq \sf TSc_{\alpha}$ and 
amalgams  found by theorem \ref{sc} are necessarily of characteristic $k$ since
embeddings preserve equations.
\end{proof}
Below in corollary \ref{infinitevarieties} we give infinitely many varieties that do not have $AP$ but their simple algebras do,
hence, in view of the coming theorem \ref{weakIP},  their free algebras have $WIP$.

Let $L$ be a signature, and let $Y$ be a set of variables.  
Then ${\sf Tm}(Y)$ denotes the absolutely free algebra of this signature on the set $Y$;
the  term algebra. Its elements are terms having variables from $Y$. ${\sf Eq}(Y)$ 
is the set of all ordered pairs of terms called {\it equations} written as $\alpha=\beta$ 
and denoted by $\epsilon, \delta.$ The variables occuring in a term, equation or set of equations $S$ is denoted
by ${\sf Var}(S)$. Given a variety $V$ of algebras in this signature, the free 
algebra of $V$ on a set $Z$ of free generators is denoted by 
$\Fr_Z(V).$

Let $\sf K$ be a class of algebras of the same signature and $Y$ be an arbitrary set of variables. 
For any $\Sigma\cup \{\epsilon\}\subseteq {\sf Eq}(Y)$, we write
$\Sigma\models_{\sf K}^Y\epsilon$, or simply $\Sigma\models_{\sf K} \epsilon$ and sometimes only $\Sigma\models \epsilon$, 
 iff for all $\A\in \sf K$ and every homomorphism $h :{\sf Tm}(Y)\to \A$, if $\Sigma\subseteq ker \phi$ then $\epsilon \in ker \phi$. 
This is a {\it substitution invariant consequence relation} as defined in \cite{george}, and if ${\sf K}$ happens to be  a variety
then it is also {\it finitary}; $\Sigma\models \epsilon$ iff $\Sigma'\models \epsilon$ for
some finite $\Sigma'\subseteq \Sigma$.

Let $Z$ be a set of variables and ${\sf Tm}(Z)$ be the absolutely free algebra over $Z$ of type $\sf TCA_{\alpha}$.
Let $h:\Tm(Z)\to \Fr_{Z}{\sf RTCA}_{\alpha}$ be the natural map, and denote $h(\alpha)$ by $\bar{\alpha}$.
For an equation $\epsilon$ of the form $\alpha=\beta$, we write $\bar{\epsilon}$ for $\bar{\alpha}=\bar{\beta}$
and we write $\bar{\Sigma}$ for $\{\bar{\epsilon}: \epsilon\in \Sigma\}.$

\begin{definition}\cite{george}
A variety $V$ has the {\it maximal Pigozzi property}, $PP_m$ for short, 
if for any sets $Y$ and $Z$
whenever
\begin{enumroman}
\item $Y\cap Z\neq \emptyset$,

\item  $\Theta_Y\in \Co(\Fr_Y(V))$ and $\Theta_Z\in \Co(\Fr_Z(V));$ are maximal congruences

\item $\Theta_Y\cap \Fr_{Y\cap Z}(V)^2=\Theta_Z\cap \Fr_{Y\cap Z}(V)^2$, then

$\Theta_Y$ and $\Theta_Z$ have a common extension to $\Fr_{Y\cup Z}(V).$
\end{enumroman}
\end{definition}

In some cases, properties of free algebras may be expressed as properties of the equational consequence relations of
the variety, as we proceed to show:

\begin{definition}\cite{george} A variety $V$ has the {\it maximal Robinson property} $RP$ if for each set of variables $Y,$ whenever
\begin{enumroman}
\item  $\Sigma \cup \Pi\cup \{\epsilon\}\subseteq {\sf Eq}(Y) $ and ${\sf Var}(\Sigma)\cap {\sf Var}(\Pi\cup \{\epsilon\})\neq \emptyset;$

\item  $\Sigma\models _V \delta$ iff $\Pi\models _V \delta$, for all $\delta\in \sf Eq(Y)$ satisfying

${\sf Var}(\delta)\subseteq {\sf Var}(\Pi)\subseteq {\sf Var}(\Sigma);$

\item  ${\sf Var}(\epsilon)\subseteq {\sf Var}(\Sigma);$

\item $\Sigma\cup \Pi\models_V \epsilon;$
\item The congruences generated by $\bar{\Sigma}$ and $\bar{\Pi}$ are maximal,
\end{enumroman}
then $\Sigma\models \epsilon.$

\end{definition}

\begin{theorem}\label{weakIP}
The following conditions are equivalent for a subvariety $V$ of $\sf TCA_{\alpha}$, $\alpha>1$:
\begin{enumarab}
%\item Semisimple algebras have $TIP$
\item Semisimple algebras in $V$ have $AP$.
\item Simple algebras have in $V$  $AP$.
\item Free $V$ algebras have $CP$ with respect to maximal congruences.
\item Free $V$ algebras have $WIP$.
\item $V$ has $PP_m.$
\item $V$ has $RP_m.$ 
\end{enumarab}
\end{theorem}
\begin{proof}
(1) implies (2) is trivial. (2) implies (1) by the argument used in the last part of theorem \ref{sc}.
(2) equivalent to (3) can be proved by using exactly the above argument using maximal congruences in place of congruences, hence 
amalgams will be {\it simple}
algebras in place of algebras.

Now we prove that (3) is equivalent to (4). This is proved for cylindric algebras in \cite{logica}. 
%The remaining part of the proof is similar to the proof of  \cite[Theorem 2.1.13 ]{P}.
Assume $CP$ relative to $U$, where $U$ is the set of proper
maximal filters in subalgebras of $\A$.
Let $X_1, X_2\subseteq \A$, and
$x\in \Sg^{\A}X_1$ and $z\in \Sg^{\A}X_2$,
such that $x\leq z$ and assume for contradiction
that there is
there is no $y$ and no finite $\Gamma\subseteq \alpha$,
such that ${\sf q}_{(\Gamma)}x\leq y\leq {\sf c}_{(\Gamma)}z$.

Then ${\sf q}_{(\Gamma)}x^\cdot -y>0$ or ${\sf q}_{(\Gamma)}[-z]\cdot y>0$
whenever $y\in \Sg^{\A}(X_1\cap X_2)$.

Hence for any finite subsets $\Delta, \theta$ of $\alpha$,
we have $u\cdot w>0$ for all $u, w\in \Sg^{\A}(X_1\cap X_2)$ such
that $u\geq {\sf q}_{(\Delta)}x$ and $w\geq {\sf q}_{(\theta)}[-z].$
Let $$P=\Fl^{\Sg^{\A}(X_1\cap X_2)}[[(\Fl^{\Sg^{\A}X_1}\{x\}\cap \Sg^{\A}(X_1\cap X_2)]\cup \Fl^{\Sg^{\A}X_2}\{-z\}\cap \Sg^{\A}(X_1\cap X_2]].$$

Then $P$ is proper, so let $P'$ be a maximal proper
filter in $\Sg^{\A}(X_1\cap X_2)$ containing
$P$. Then there are maximal filters $M$ of $\Sg^{\A}X_1$ and $N$ of $\Sg^{\A}X_2$
such that (*)
$$\Fl^{\Sg^{\A}X_1}\{x\}\subseteq M\text {  and }\Fl^{\Sg^{\A}X_2} \{-z\}\subseteq N,$$
and $M\cap \Sg^{\A}(X_1\cap X_2)=P'=N\cap \Sg^{\A}(X_1\cap X_2)$.
By assumption, we have
$\Fl^{\A}(M\cup N)$ is proper, and
so it is not the case that $x\leq z,$ for if $x\leq z$, then $x\cdot (-z)=0$ and so by (*) 
we get $0\in \Fl^{\A}(M\cup N)$ 
and so $\Fl^{\A}(M\cup N)=\A$.
This is a contradiction and we are done.

For the converse. Assume that $\A$ has $WIP$.
Let $M$ be a filter of $\Sg^{\A}X_1$ and $N$ be a filter
of $\Sg^{\A}X_2$, both maximal,  such that
$M\cap \Sg^{\A}(X_1\cap X_2)=N\cap \Sg^{\A}(X_1\cap X_2).$

Assume for contradiction that $\Ig^{\A}(M\cup N)=\A$.
Then there exist $x\in M$, $z\in N$ such that $x\cdot z=0$. By assumption there is an element
$y\in \Sg^{\A}(X_1\cap X_2)$ and a finite $\Gamma\subseteq \alpha$
such that ${\sf q}_{(\Gamma)}x\leq y\leq {\sf c}_{(\Gamma)}z$,
hence $y\in \Fl^{\Sg^{\A}X_1}\{x\}$ and $-y\in \Fl^{\Sg^{\A}X_2}\{z\}$,
and so $-y\in M\cap \Sg^{\A}(X_1\cap X_2)$
and $y\in N\cap \Sg^{\A}(X_1\cap X_2).$
Hence $0=-y\cdot y\in M$ which is impossible.
We conclude that $\Fl^{\A}(M\cup N)$ is proper and maximal,
and it induces the required maximal congruence.

It is clear that (5) and (3) are equivalent. 
We lastly prove the equivalence of (5) and (6). Suppose $V$ has $PP_m$ and that conditions (i) (ii) (iii), and (iv) are satisfied for the 
$RP_m$. Let $Y={\sf Var}(\Sigma)$ and $Z={\sf Var}(\Pi)$.
Let $\Theta_Y$ be the congruence generated by $\bar{\Sigma}$ in $\Fr(Y)$ and $\Theta_Z$ be the congruence 
generated by $\bar{\Pi}$ in $\Fr(Z)$.
Then both are maximal congruences and $\Theta_Y\cap \Fr(Y\cap Z)^2=\Theta_Z\cap \Fr(Y\cap Z)^2$. Hence by $PP$ there exists
$\Theta\in \Co(\Fr(Y\cup Z))$ such that $\Theta_Y=\Theta\cap \Fr(Y)^2$ and $\Theta_Z=\Theta\cap \Fr(Z)^2$. 
We may assume that $\Theta$ is the congruence generated by $\Theta_Y\cup \Theta_Z$ in 
$\Fr(Y\cup Z)$. By (iv) we have $\bar{\epsilon}\in \Theta$. But ${\sf Var}(\epsilon)\subseteq Y,$ we have
$\bar{\epsilon}\in \Theta_Y$ and
$\Sigma\models \epsilon$.

Conversely, assume $V$ has $RP_m$ and that conditions
(i), (ii), (iii) are satsified for the $PP_m$. Choose $\Sigma$ and $\Pi$ such
that $\Theta_Y$ is the congruence generated by $\bar{\Sigma}$ in $\Fr(Y)$ and $\Theta_Z$ is
the congruence generated by $\Pi$ in $\Fr(Z)$. Then (i) and (ii) of the $RP_m$ hold. Let
$\Theta$ be the congruence generated by $\Theta_Y\cup \Theta_Z$ in $\Fr(Y\cup Z)$ 
is as required.
\end{proof}
The natural question at this point is. What does the {\it usual } interpolation property correspond to. On the global level it corresponds to the 
super amalgamation property. One implication can be distilled without much effort from the proof theorem
\ref{s}. The other direction, that is $SUPAP$ implies $IP$ in free algebras is proved by Madarasz and Maksimova \cite{Mad, Mak},
in a more general setting, of which Boolean algebras with operators, and cylindric algebras with interior operators,
are a special case. 
\begin{theorem} Let $\alpha$ be an ordinal $>0$.
Let $V$ be a variety of $\sf TCA_{\alpha}$. Then the following conditions are equivalent:
\begin{enumarab}
\item $V$ has $SUPAP$
\item The free algebras have $IP.$
\end{enumarab}
\end{theorem}

If $\alpha<\omega$, then $\sf TCA_{\alpha}$ is a discriminator variety, 
with discriminator term ${\sf c}_{(\alpha)}$; in particular, every subdirectly indecomposable 
algebra is simple and hence every algebra is semisimple. This gives:

\begin{theorem} If $\alpha<\omega$ then for any subvariety $V$ of $\sf TCA_{\alpha}$ 
all conditions in theorems \ref{app} and \ref{weakIP} are equivalent:
\end{theorem}
\begin{proof} It suffices to show that free algebras have $WIP$ implies free algebras have $AIP$. 
Let $\A=\Fr_{\beta}V$ $X_1, X_2\subseteq \A$, 
$a\in \Sg^{\A}X_1$ and $b\in \Sg^{\A}X_2$ such
that $a\leq b$. Then $a\leq {\sf c}_{(\alpha)}a\leq {\sf c}_{(\alpha)}b$. Hence by $WIP$ there 
exists $d\in \Sg^{\A}(X_1\cap X_2)$ and $\Gamma\subseteq \alpha$ such that 
$a\leq {\sf c}_{(\Gamma)}^{\partial}{\sf c}_{(\alpha)}a\leq d\leq {\sf c}_{(\Gamma)}{\sf c}_{(\alpha)}b={\sf c}_{(\alpha)}b$, hence $\A$ has $AIP$
and we are done. Notice that if $\alpha$ is the dimension $<\omega$ then both $AIP$ and $WIP$ are equivalent to;
with $\A$, $a, b$ as above; that there exists $d\in \Sg^{\A}(X_1\cap X_2)$ such that $a\leq d\leq {\sf c}_{(\alpha)}b$. 
\end{proof}
We will see in corollary \ref{notequivalent} that the above 
theorem is is not true for the infinite dimensional case. Indeed for $\alpha\geq \omega$, $\sf TCA_{\alpha}$ is not a discriminator variety;
subdirectly indecomposable algebras that are not simple can be esaily constructed. 
We know from theorem \ref{typeless} that the free algebras 
in any variety $V$ containing the representable algebras have $WIP$, but we will see in theorem \ref{am}
that $\sf RTCA_{\alpha}$ does not have $AP$, hence by theorem \ref{app} the 
free algebras  do not have $AIP.$ 

Now we define yet other {\it restricted} interpolation properties, that are the adaptation of Pigozzi's restricted forms of interpolation
defined for cylindric algebras, to our present context. Here we look for the interpolant in the minimal subalgebra of the free algebra.
From the logical point of view the formulas to be interpolated contain only the equality symbol as a common symbol, so we are 
looking for an interpolant that contains no other symbols, we are looking for a formula built up only of equations, 
that is their atomic subformulas are of the form $x_i=x_j$ $(i,j\in \omega)$,
where $x_i$ and $x_j$ are variables;  reflected algebraically by the 
diagonal element ${\sf d}_{ij}$.

\begin{definition}\label{em}
Let $\A\in \TCA_{\alpha}$.
\begin{enumarab}
\item $\A$ has the {\it restricted interpolation property}  if whenever $x\leq z$, $x\in \Sg^{\A}Y$ and $z\in \Sg^{\A}Z$, with $Y\cap Z=\emptyset$,
then there exists $y\in \Sg^{\A}(Y\cap Z)$ such that either $x\leq y\leq z.$
\item $\A$ has the {\it almost restricted interpolation property} if   whenever $x\leq z$, as in the previous item,
then there exist $y\in \Sg^{\A}(Y\cap Z)$ and a  finite $\Gamma\subseteq \alpha$, such that ${\sf q}_{(\Gamma)}x\leq y\leq z$.
\item $\A$ the {\it weak restricted interpolation property} if whenever $x\leq z$ as in the previous item,
then there exist $y\in \Sg^{\A}(Y\cap Z)$ and a finite $\Gamma\subseteq \alpha$
such that ${\sf q}_{(\Gamma)}x\leq y \leq {\sf c}_{(\Gamma)}z=1.$
\end{enumarab}
\end{definition}
In the above definition the algebra $\Sg^{\A}(Y\cap Z)$ is a {\it minimal} algebra, it is generated by the diagonal elements, and it has no
{\it proper} subalgebras. Let $\sf TMn_{\alpha}$ denote the class of such minimal algebras, namely, algebras with no 
proper subalgebras. Clearly $\sf TMn_{\alpha}\subseteq \sf TDc_{\alpha}$ 
for infinite $\alpha$. This simple obervation will be used in 
the coming proof.

\begin{theorem}\label{weak2}  Let $\alpha$ be an infinite ordinal.
Let $\beta$ be any cardinal $>0$. Then $\Fr_{\beta}\sf RTCA_{\alpha}$
has the strong restricted interpolation property.
\end{theorem}
\begin{demo}{Proof}
Let $\A=\Fr_{\beta}\sf TRCA_{\alpha}$ and let $X_1, X_2\subseteq \beta$ be disjoint sets.  We can assume without loss of generality
that $X_1\cup X_2=\beta.$
Assume that $a\in \A_1=\Sg^{\A}X_1$ and $b\in \A_2=\Sg^{\A}X_2$ such that
$a\leq b$. Since $X_1\cap X_2=\emptyset$, we have
$\A_0=\Sg^{\A}(X_1\cap X_2)=2$ embeds in $\Sg^{\A}X_1$ and $\Sg^{\A}X_2$,
respectively via the inclusion maps $i_0$ and $i_1$ say.
From theorem \ref{SUPAP}
that there is a ${\D}\in \sf RTCA_{\alpha}$,
a monomorphism $m_1$ from ${\A}_1$ into ${\D}$ and a monomorphism $m_2$ from
${\A}_2$ into $\D$ such that $m_1\circ i_1=m_2\circ i_2$, and
$(\forall x\in A_j)(\forall y\in A_k)
(m_j(x)\leq m_k(y)\implies (\exists z\in A_0)(x\leq i_j(z)\land i_k(z) \leq y))$
where $\{j,k\}=\{1,2\}.$
Now since $\A$ is free, there exists a homomorphism  $f:\A\to \D$ such that $f\upharpoonright \A_1=m_1$ and $f\upharpoonright \A_2=m_2$.
Since $f(a)\leq f(b)$ it follows that $m_1(a)\leq m_2(b)$.
Hence there exists $z\in \A_0$ such that $a\leq z\leq b$. 
\end{demo}

Before stating our next result, we need:
\begin{definition}
$\sf K$  has {\it the (strong) embedding property} if it has the (strong) amalgamation property when the base algebra is 
is minimal.
\end{definition}

\begin{corollary}\label{embedding} Let $\alpha$ be an infinite ordinal. Then $\sf RTCA_{\alpha}$ has the strong embedding property.
Furthermore, if the algebras $\A$ and $\B$ to be amalgamated
(agreeing on their minimal subalgebras)
are simple, semi-simple or in $\Sc_{\alpha}$
then so is the
amalgam.
\end{corollary}
\begin{proof} The first part is from the proof of theorem \ref{weak2}, using theorem \ref{SUPAP}.
The second part follows from the arguments used in theorem \ref{sc}.
\end{proof}
We will see in a while that for finite $\alpha>1$, the situation is different. 
The class $\sf RTCA_{\alpha}$ does not have the embedding property and $\Fr_{\beta}\sf K$
for any class $\sf K$ between $\sf RTCA_{\alpha}$ and $\sf TCA_{\alpha}$ does 
not have the weakest restricted interpolation property.  

\section{Negative results}

Now that we have obtained such equivalences, the natural question is how far can we get as far as interpolation 
is concerned with the free representable algebras.
We know that they enjoy the weak interpolation property. In the next theorem 
we show that the representable algebras does not have  
$AP$. It is known \cite{P} that the class $\sf RCA_{\alpha}$ for infinite $\alpha$  does not have $AP$, so obtain an analogous result by bouncing it to
the cylindric case.

\begin{theorem}\label{am} Let $\alpha\geq \omega$. Then any class $\sf K$ such that $\TCA_{\alpha}\subseteq \sf K\subseteq 
\sf TRCA_{\alpha}$ does not have $AP.$
\end{theorem}
\begin{proof} Take  $\A_0, \A_1, \A_2$
 in $\sf RCA_{\alpha}$ and $f:\A_0\to \A_1$, and  $g:\A_0\to \A_1$ injective homomorphisms
for which there are no
$\D\in \sf RCA_{\alpha}$ and injective homomorphisms $m:\A_1\to \D$  $n:\A_2\to \D$ such that $m\circ f=n\circ g$,
endow the base of each the interior topology stimulating the interior operators as identity functions,
and so finding an amalgam for the resulting $\sf RTCA_{\alpha}$s (with same embedding maps), 
will give an amalgam to the original $\sf RCA_{\alpha}$s, by taking its 
$\sf CA$ reduct \cite{P, amal} which is a contradiction.
We explicitly describe such algebras.
Let $\A=\Fr_4\sf CA_{\alpha}$ with $\{x,y,z,w\}$ 
its free generators. Let $X_1=\{x,y\}$ and $X_2=\{x,z,w\}$.
Let $r, s$ and $t$ be defined as follows:
\begin{align*}
r& = {\sf c}_0(x\cdot {\sf c}_1y)\cdot {\sf c}_0(x\cdot -{\sf c}_1y),\\
s& = {\sf c}_0{\sf c}_1({\sf c}_1z\cdot {\sf s}^0_1{\sf c}_1z\cdot -{\sf d}_{01}) + {\sf c}_0(x\cdot -{\sf c}_1z),\\
t &= {\sf c}_0{\sf c}_1({\sf c}_1w\cdot {\sf s}^0_1{\sf c}_1w\cdot -{\sf d}_{01}) + {\sf c}_0(x\cdot -{\sf c}_1w),
\end{align*}
where $ x, y, z, \text { and } w$ are the first four free generators
of $\A$.
Then $r\leq s\cdot t$.
Let $\D=\Fr_4\sf RCA_{\alpha}$ with free generators $\{x', y', z', w'\}$.
Let  $\psi: \A\to \D$ be defined by the extension of the map $t\mapsto t'$, for $t\in \{x,y,x,w\}$.
For $i\in \A$, we denote $\psi(i)\in \D$ by $i'$.
Let $I=\Ig^{\D^{(X_1)}}\{r'\}$ and $J=\Ig^{\D^{(X_2)}}\{s'\cdot t'\}$, and let
$$L=I\cap \D^{(X_1\cap X_2)}\text { and }K =J\cap \D^{(X_1\cap X_2)}.$$
Then $L=K$, and $\A_0=\D^{(X_1\cap X_2)}/L$  can be embedded into
$\A_1=\D^{(X_1)}/I$ and $\A_2=\D^{(X_2)}/J$,
but there is no amalgam even in $\sf CA_{\omega}$ \cite[Theorem 4.1]{conference}. 
\end{proof}

We have a stronger result for finite dimensions:

\begin{theorem}\label{Comer}
Let $n$ be finite $>1$. Then any class $\sf K$ between $\sf TRCA_{n}$ and $\TCA_n$ does not have $EP$.
Furthermore, the algebras witnessing failure of $EP$ can be chosen to be 
set  algebras hence are simple.
\end{theorem}
\begin{proof} Assume first that $n>1$. Let $1<n\leq |U_0|<|U_1|$ and $|U_0|<\omega$.
Let $\A_i=\A(n, U_i)$ be the set algebra with unit $^{n}U_i$ aand universe $\wp(^{n}U_i)$ 
where $U_i$ has the discrete topology.
Let $D_i$ be the principal diagonal in $\A_i$. That is $D_i=\prod_{k,l<n}{\sf d}_{kl}.$ 
Let $\mathfrak{M}$ be the minimal subalgebra of $\A_0$. Then $\mathfrak{M}$ is embeddable in $\A_0$ and $\A_1$ 
via $g_0$ and $g_1$
such that $g_0\circ g^{-1}$ is an isomorphism from $g_0\mathfrak{M}$ and $g_1\mathfrak{M}$ and  
$g_1g_0^{-1}D_0=D_1.$
Here we are using that the minimal subalgebras of $\A_0$ and $\A_1$ are isomorphic, 
and they remain so after endowing their bases with the discrete 
topology. 
This follows from the fact that they are both simple and have characteristic zero \cite[2.5.30]{HMT1}.
Then, as proved in \cite{Comer} there can be no 
$\B\in \TCA_{n}$, $f_0:\A_0\to \B'$ and $f_1:\A_1\to \B'$ injective homomorphisms such that $f_1\circ g_1=f_0\circ g_0$.  

\end{proof}
Note that if $n=1$ and we drop the {\it cross axiom} ${\sf s}_i^jI(i)=I(j){\sf s}_i^j$, then the same set algebras 
can be viewed as one dimensional algebras, with $I_0$ interpreted 
like the second cylindrifier, which means that $EP$ fails for this class of one dimensional 
algebras.

An algebra in $\A\in \sf TCs_{\alpha}$ is called a {\it full set algebra} if the universe of $\A$ is $\wp(^{\alpha}U)$ for some set
$U$, that is, it consists of all subsets of $^{\alpha}U$. 

\begin{corollary}\label{notequivalent} For $\alpha\geq \omega$ any subvariety of 
$\sf TCA_{\alpha}$ containing the class  of all full set algebras of dimension $\alpha$ fails all the 
conditions of theorem  \ref{app}, but satisfies all conditions of theorem \ref{weakIP}. If $\alpha$ is finite $>1$, any  variety 
containing the class of all full set algebras of dimension $\alpha$ fails
all conditions of theorem \ref{weakIP}, hence also those in \ref{app}.
\end{corollary}
\begin{proof} From theorems \ref{am} and \ref{Comer}.
\end{proof}

%\begin{corollary} For $1<\alpha<2$, $\sf RTCA_{\alpha}$ does not have the embedding property, 
%hence the free algebras do not have restricted $IP$.
%\end{corollary}

Next we give a categorical formulation to the $UNEP$. We review some categorical concepts from \cite{cat}.
For a category
$L,$ $Ob(L)$ denotes the class of objects of the category
and $Mor(L)$ denotes the corresponding class of morphisms.

\begin{definition} Let $L$ and $K$ be two categories.
Let $G:K\to L$ be a functor and let $\B\in Ob(L)$. A pair $(u_B, \A_B)$ with $\A_B\in Ob(K)$ and $u_B:\B\to G(\A_B)$ is called a universal map
with respect to $G$
(or a $G$ universal map) provided that for each $\A'\in Ob(K)$ and each $f:\B\to G(\A')$ there exists a unique $K$ morphism
$\bar{f}: \A_B\to \A'$ such that
$$G(\bar{f})\circ u_B=f.$$
\end{definition}

\begin{displaymath}
    \xymatrix{
        \mathfrak{B} \ar[r]^{u_B} \ar[dr]_f & G(\mathfrak{A}_\mathfrak{B}) \ar[d]^{G(f)}  &\mathfrak{A}_\mathfrak{B} \ar[d]^{\hat{f}} \\
             & G(\mathfrak{A}')  & \mathfrak{A}'}
\end{displaymath}

Viewing the neat reduct operator as a {\it functor} introduced next, is a theme initiated in \cite{conference}, wrapping up deep results
on the amalgamation property for both cylindric and polyadic algebras in the language
of arrows.
\begin{definition}
The {\it neat reduct functor},  $\Nr$ for short, is defined from $\bold K=\{\A\in \TCA_{\alpha+\omega}: \Sg^{\A}\Nr_{\alpha}\A=\A\}$ to
$\sf RTCA_{\alpha}$ by sending every object $\A\in \bold K$, to $\Nr_{\alpha}\A$, and sending injective homomorphisms
to their restrictions, that is for $\A, \B\in \bold \K$ and $f:\A\to \B$, an injective homomorphism,
$\Nr(f)=f\upharpoonright \Nr_{\alpha}\A$. 
\end{definition}
It is clear that $f(\Nr_{\alpha}\A)\subseteq \Nr_{\alpha}\B,$ hence
this functor is well defined.  Here we are restricting morphisms to only injective homomorphisms.

For a class $\sf K$, we let ${\sf Sim}(\sf K)$ denotes the class of simple algebras in $\sf K$.
\begin{corollary}\label{infinitevarieties}
Let $\alpha$ be an infinite ordinal. 
\begin{enumarab}
\item For any $k\geq 4$, and any class $\sf K$ such that $\sf RTCA_{\alpha}\subseteq \sf K\subseteq \TCA_{\alpha}$,
$\Fr_k\sf K$ does not have $AIP$ but has $WIP$.
\item $\sf RTCA_{\alpha}$ does not have $UNEP$.
\item For any $k\geq 0$, the variety $S\Nr_{\alpha}\sf TCA_{\alpha+k}$ does not have $AP$, hence it does not have $MIP$, indeed
it does not have any of the properties in theorem \ref{app} but it has $WIP$, hence satisfy 
all properties in theorem \ref{weakIP}.
\item There is an $\A\in \sf TCA_{\alpha}$ that does not have a universal map
with respect to the functor $\Nr.$
\item $\Nr$ does not have a right adjoint.
\end{enumarab}
\end{corollary}
\begin{proof}
\begin{enumarab}
\item  By theorem \ref{typeless} and \ref{am}. 
\item Let $\A_0$ be as in the proof of theorem \ref{am}, so that $\A_0$ is not in the amalgamation base of $\sf RTCA_{\alpha}$.
If $\A_0$ has $UNEP$ then using the same reasoning in theorem \ref{SUPAP}, by appeal to theorem \ref{s}, 
$\A_0$ would be  in the amalgamation base of $\sf RTCA_{\alpha}$ which is impossible.
\item The algebras $\A_0$, $\A_1$, $\A_1$ in the proof of theorem \ref{am} are representable, 
hence they are in $\bold S\Nr_{\alpha}\sf TCA_{\alpha+\omega}$
but they do not have an amalgam in $\sf TCA_{\alpha}$, {\it a fortiori} in 
$\bold S\Nr_{\alpha}\sf TCA_{\alpha+k}$. $WIP$ follows from theorem \ref{typeless} or by noting that
${\sf Sim}(\TCA_{\alpha})={\sf Sim}(\sf RTCA_{\alpha})={\sf Sim}(\bold S\Nr_{\alpha}\sf TCA_{\alpha+k})$
and then using theorem \ref{weakIP}.

\item Let $\A_0$ be as in the proof of theorem \ref{am}. Then $\A_0$ does not have $UNEP$. 
So $\A_0$  generates non - isomorphic algebras in extra dimensions, 
the existence of a universal map for it,
will force that these algebras are actually isomorphic, fixing it pointwise, and this cannot happen.

In more detail, assume that $\A_0$ neatly embeds into $\B$ via $e_B$ and into
$\B'$ via $e_{B'}$. Let $(e, \C)$ be a universal map for $\A_0$, so that
$\A$ neatly embeds into $\C$ via $e$. (See the above diagram). By universality, there exists isomorphisms $f:\C\to \B$ and $k:\C\to \B'$ such that
$f\circ e=e_B$ and $k\circ e=e_B'$. The maps are injective by definition,
they are surjective, because $\A_0$ is contained in $\C$ and it generates both $\B$ and $\B'$.
We infer that $\B$ and $\B'$ are isomorphic, but we want more. We want to exclude special isomorphisms
(in principal, isomorphisms can exist as long as they do not fix
$\A$ pointwise). We have $h=k\circ f^{-1}:\B\to \B'$ is an isomorphism such that
$h\circ e_B=e_{B'}$, and this isomorphism is as required, leading to a  contradiction.

\item It is known \cite[Theorem 27.3 on  p. 196]{cat}
 that if $G:K\to L$ is a functor such that  each $\B\in Ob(K)$ has a $G$ universal map $(\mu_B, \A_B)$, 
then there exists a unique adjoint situation $(\mu, \epsilon):F\to G$
such that $\mu=(\mu_B)$ and for each $\B\in Ob(L),$
$F(\B)=\A_B$. Conversely, if we have an adjoint situation $(\mu,\epsilon):F\to G$ 
then for each $\B\in Ob(L)$ $(\mu_B, F(\B))$ have a $G$ universal map.

\end{enumarab}

\end{proof}

\begin{theorem} For any $\alpha>0$, the following conditions are equivalent for $\A\in \sf TCA_{\alpha}$.
\begin{enumarab}
\item $\A$ has the $UNEP.$
\item $\A\in {\sf APbase}({\sf RTCA}_{\alpha}).$
\item $\A$ has  a universal map with respect to the functor $\Nr.$
\end{enumarab}
\end{theorem}
\begin{proof} This is proved for cylindric algebras in \cite{logica}. The proof lifts with no modifications.
\end{proof}

We need another algebraic counterpart of a yet another
definability property, namely, Beth definability.  A class
$\sf K$ {\it has $ES$ if epimorphisms in $\sf K$ (in the categorical sense)
are surjective.}

The next theorem is  folklore \cite{AUU, MStwo, MS}.

\begin{theorem}\label{e} Let $\sf K$ be a class of algebras. Then
\begin{enumarab}
\item If ${\sf K}$ has $SUPAP$ then it has $SAP$, $AP$ and $ES.$
\item ${\bf SP}{\sf K}$ has $SAP$ if and only if it has $AP$ and $ES$.
\end{enumarab}
\end{theorem}

Our next $ES$ result can also be easily distilled from the cylindric counterpart, as follows.

\begin{theorem}\label{es}
Let $\alpha>1$. Then for any class $\sf K\subseteq \TCA_{\alpha}$, that contains the class 
of simple representable algebras, $ES$, hence $SAP$ fail. In particular, for $\alpha\geq \omega$ 
and $k\geq 1$, $ES$ fails for $S\Nr_{\alpha}\sf TCA_{\alpha+k}.$
\end{theorem}
\begin{proof} 
Assume that $\alpha\geq \omega$. Notice that in this case all simple algebras are representable.
Let $\A, \B$ be the algebras such that $\A\subseteq\B$ is an epimorphism that is not surjective.
Such algebras are constructed in \cite{MStwo}; they are weak set algebras,
$\A$ is generated by a single element $R$ and $\B$ is generated by two elements $R$ and $X$ hence 
$X\notin \A$. Also  $\A$ and $\B$ have the same unit  $V$, with common base $U$
tha can be endowed with the discrete topology. So we can assume that $\A,\B\in \sf RTCA_{\alpha}$
and still we have $\A\subseteq \B$.  Assume that $\D\in \TCA_{\alpha}$
witnesses that this inclusion is not an epimorphism.
Hence there is $\D\in \sf TCA_{\alpha}$ and homomorphisms $f,g:\B\to \D$
such that $f(X)\neq g(X)$, this is impossible because then $\Rd_{ca}\D$
witnesses that $\A\subseteq \B$ is not an epimorphism. 

In \cite{MStwo} 
it is shown that $\A$ and $\B$ can be chosen to be semi-simple and in \cite{es} it is shown that they can further
be chosen to be simple. For finite $\alpha>1$ the required follows easily from the construction in \cite{ACNS} expanding, like we often did before,
the constructed algebras by the identity interior operators corresponding to the 
discrete topology. The last part follows from the fact that $\sf RTCA_{\alpha}\subseteq S\Nr_{\alpha}\CA_{\alpha+k}\subseteq \sf TCA_{\alpha}$ 
for all $\alpha$.
\end{proof}
\begin{theorem}\label{simon}
\begin{enumarab}
\item  Let $\D$ the full set algebra with unit $^{\omega}\omega$; with $\omega$ having 
the discrete topology. Let $\M$ be its minimal subalgeba.
Then $\M\in \sf SUPAPbase(\sf RTCA_{\omega})\sim \sf APbase(\TCA_{\omega}).$
Furthermore,  $\TCA_{\omega}$ does not have $EP$, and $\Fr_{\omega}\CA_{\omega}$ does not have weak 
restricted $IP.$ 

\item If $1<\alpha<\omega$, then for any $\sf K$ 
such that $\sf RTCA_{\alpha}\subseteq \sf K\subseteq \TCA_{\alpha}$,
$\Fr_{\omega}\sf K$ does not have weak restricted $IP$.
\end{enumarab}
\end{theorem}
\begin{proof} For the infinite dimensional case witness \cite{Sayedneat, Simon}. 
The second  part follows from the reasoning in theorem \ref{weakIP}.  The second item concerning finite dimensional algebras 
follows from theorem \ref{Comer}.
\end{proof}

\section{Representability and amalgamation for various topological polyadic algebras}

The class ${\sf TPCA}_{\alpha}$, to be dealt with next, 
is defined by restricting the signature and axiomatization of Halmos' 
polyadic algebras to finite cylindrifiers,  so that we have all substitition operators but only ${\sf c}_i$ for $i\in \alpha$, and interior operators $I_i$.
Here we do not have diagonal elements and we consider {\it only infinite dimensions}.
In more detail:
The dimension set of $x$, in symbols $\Delta x$, is defined exactly as in the $\sf TCA$ case; that is
$\Delta x=\{i\in \alpha: {\sf c}_ix\neq x\}$.
If $\A\in \sf TPCA_{\alpha}$, then $\Rd_{pa}\A$ denotes its reduct obtained by discarding all interior operators.
Henkin ultrafilters are defined exactly like before; they are the ultrafilters that {\it eliminate cylindrifiers.}
That is for $\A\in \sf TPCA_{\alpha}$ a Boolean ultrafilter $F$ is {\it Henkin} if for all $x\in \A$ for all $k<\alpha$, if 
${\sf c}_kx\in F,$ then there exists $l\notin \Delta x$, 
such that ${\sf s}^k_lx\in F$.

\begin{theorem}\label{dilation} Let $\alpha<\beta$ be infinite ordinals. 
Then for every $\A\in \sf TPCA_{\alpha}$ there is a 
unique $\B\in \sf TPCA_{\beta}$ up to isomorphism that fixes $\A$ such that
$\A\subseteq \Nr_{\alpha}\B$ and for all $X\subseteq A$, $\Sg^{\A}X=\Nr_{\alpha}\Sg ^{\B}X$. 
In particular, $\TPCA_{\alpha}$ has $NS$ and $UNEP$.
Furthermore, if $F$ is a Henkin  ultrafilter of $\B$ and $a\in F$, then 
there exists a topology on $\beta$ and a homomorphism 
$f:\A\to (\wp({}^{\alpha}\beta), J_i)_{i<\alpha}$ with $f(a)\neq 0$ 
where  for $i<\alpha$ and $X\subseteq {}^{\alpha}\beta$, 
$J_iX=\{s\in {}^{\alpha}\beta: s_i\in {\sf int}\{u\in U: s^i_u\in X\}\}$. 
\end{theorem}
\begin{proof} Dilations are proved to exist similar to the arguments used in  \cite{DM, conference}; the rest follows 
using analogues of \cite[Theorems 2.9, 3.6]{part1}.
Nevertheless forming dilations here is more involved. 
We extensively use the techniques in \cite{DM}, but we have to watch out, for we only have finite cylindrifications.
Let $(\A, \alpha,S)$ be a transformation system in the sense of \cite{DM}.
Substitutions in $\A$, induce a homomorphism
of semigroups $S:{}^\alpha\alpha\to End(\A)$, via $\tau\mapsto {\sf s}_{\tau}$.
The operation on both semigroups is composition of maps; the latter is the semigroup of endomorphisms on
$\A$. For any set $X$, let $F(^{\alpha}X,\A)$
be the set of all functions from $^{\alpha}X$ to $\A$ endowed with Boolean  operations defined pointwise and for
$\tau\in {}^\alpha\alpha$ and $f\in F(^{\alpha}X, \A)$, put ${\sf s}_{\tau}f(x)=f(x\circ \tau)$.
This turns $F(^{\alpha}X,\A)$ to a transformation system as well.
The map $H:\A\to F(^{\alpha}\alpha, \A)$ defined by $H(p)(x)={\sf s}_xp$ is
easily checked to be an embedding. Assume that $\beta\supseteq \alpha$. Then $K:F(^{\alpha}\alpha, \A)\to F(^{\beta}\alpha, \A)$
defined by $K(f)x=f(x\upharpoonright \alpha)$ is an embedding, too.
These facts are straightforward to establish, cf.
\cite[Theorems 3.1, 3.2]{DM}.
Call $F(^{\beta}\alpha, \A)$ a minimal functional dilation of $F(^{\alpha}\alpha, \A)$.
Elements of the big algebra, or the (cylindrifier free)
functional dilation, are of form ${\sf s}_{\sigma}p$,
$p\in F(^{\beta}\alpha, \A)$ where $\sigma$ is one to one on $\alpha$, cf. \cite[Theorems 4.3-4.4]{DM}.

We can assume that $|\alpha|<|\beta|$.
Let $\B$ be the algebra obtained from $\A$, by discarding its cylindrifiers,
then dilating it to $\beta$ dimensions, that is,
taking  a minimal functional dilation in $\beta$
dimensions, and then re-defining cylindrifiers in the bigger algebra,
on the big algebra, so that they agree with their values in $\A$ as follows (*):
\begin{align*}
{\sf c}_k{\sf s}_{\sigma}^{\B}p&={\sf s}_{\rho^{-1}}^{\B} {\sf c}_{\rho(\{k\}\cap \sigma \alpha)}{\sf s}_{(\rho\sigma\upharpoonright \alpha)}^{\A}p.\\
I(k){\sf s}_{\sigma}^{\B}p&={\sf s}_{\rho^{-1}}^{\B} I(\rho(\{k\}\cap \sigma \alpha)){\sf s}_{(\rho\sigma\upharpoonright \alpha)}^{\A}p.
\end{align*}
Here $\rho$ is a any permutation such that $\rho\circ \sigma(\alpha)\subseteq \sigma(\alpha.)$
It can be checked by a somewhat tedious computation \cite{DM} that the definition is sound; in other words 
it is independent of $\rho, \sigma, p$.
and it defines the required dilation.

To prove $UNEP$ let  $\A, \A' \in \TPCA_{\alpha}$ and $\beta>\alpha$. Let
$\B, \B' \in \TPCA_{\beta}$ and assume that $e_A, e_{A'}$ are embeddings from $\A, \A'$ into   $\Nr_\alpha \B,
\Nr_\alpha \B'$, respectively, such that
$ \Sg^\B (e_A(A)) = \B$
and $ \Sg^{\B'} (e_{A'}(A')) = \B',$
and let $ i : \A \longrightarrow \A'$ be an isomorphism.
We need to ``lift" $i$ to $\beta$ dimensions.
Let $\mu=|A|$. Let $x$ be a  bijection  from $\mu$ onto $A.$
Let $y$ be a bijection from $\mu$ onto $A'$,
such that $ i(x_j) = y_j$ for all $j < \mu$.
Let $\D = \Fr_{\mu}\TPCA_{\beta}$ with generators $(\xi_i: i<\mu)$. Let $\C = \Sg^{\Rd_\alpha \D} \{ \xi_i : i < \mu \}.$
Then $\C \subseteq \Nr_\alpha \D,\  C \textrm{ generates } \D
~~\textrm{and so by the previous lemma }~~ \C =\Nr_{\alpha}\D.$
There exist homomorphisms $ f :\D\to \B$ and $f':\D$ to $\B'$ such that
$f (g_\xi) = e_A(x_\xi)$ and $f' (g_\xi) = e_{A'}(y_\xi)$ for all $\xi < \mu.$
Note that $f$ and $f'$ are both surjective.We now have
$e_A \circ i^{-1} \circ e_{A'}^{-1} \circ ( f'\upharpoonleft \C) = f \upharpoonleft \C.$
Therefore $ker f' \cap \C = ker f \cap \C.$
Hence by $NS$ we have $ker f'=\Fl(ker f' \cap \C) = \Fl(ker f \cap \C)=kerf.$

Let $y \in B$, then there exists $x \in D$ such that $y = f(x)$. 
Define $ \hat{i} (y) = f' (x).$
The map is well defined and is as required.

For the last part, assume that a Henkin ultrafilter $F$ is given and  
define  $f:\A\to \wp(^{\alpha}\beta)$ via
$$p\mapsto \{\tau \in {}^{\alpha}\beta:  {\sf s}_{\tau\cup Id_{\beta\sim \alpha}}p\in F\}.$$
Then like the proof of \cite[Lemma 3.4]{part1} $f$ preserves the polyadic operations. Handling interior operators here is easier, 
for we do not have diagonal elements, and hence
we are not forced to define a congruence on the base of the representation as done before in case of cylindric topological algebras.
For $i\in \alpha$ and $p\in \A$, let 
$$O_{p,i}=\{k\in \beta: {\sf s}_i^kI(i)p\in F\}$$
Let $${\cal B}=\{O_{p,i} : i\in \alpha, p\in A\}.$$
Then it is easy to check that ${\cal B}$ 
is the base for a topology on $\beta$.
Let $W={}^{\alpha}\beta$.
For each $i<\alpha$ 
$$J_i: \wp(W)\to \wp (W)$$
by $$x\in J_iX\Longleftrightarrow \exists U\in {\cal B}(x_i\in U\subseteq \{u\in \alpha: x^i_u\in X\}),$$
where $X\subseteq W$. 
We now check that $f$ preserves the interior operators $J(i)$ $(i<\alpha)$, too.
We need to show 
$$\psi(I_ip)=J_i(\psi(p)).$$
Let $x$ be in $\psi(l_ip)$. 
Then $x_i\in \{u: {\sf s}_u^i I(i){\sf s}_yp\in F\}\in q$
where $y:\alpha\to \beta$, $y\upharpoonright \alpha\sim \{i\}=x\upharpoonright \alpha\sim \{i\}$
and $y(i)=i$.
But $I_i{\sf s}_yp\leq {\sf s}_y p,$ hence
$$U=\{u: {\sf s}_u^i I_i{\sf s}_yp\in F\}\subseteq \{u: {\sf s}_u^i{\sf s}_yp\in F\}.$$
It follows that $x_i\in U\subseteq \{u: x^i_u\in \Psi(p)\}.$
Thus $x\in J_i\psi(p).$
The other  direction is the same as in the proof of \cite[Lemma 3.6]{part1}.

\end{proof}

\begin{corollary}\label{ns} For any ordinal $\beta>\alpha$  $\sf TPCA_{\alpha}=\bold S\Nr_{\alpha}\sf TPCA_{\beta}.$
\end{corollary}
\begin{theorem}\label{positive}.
\begin{enumarab}
\item $\Fr_{\beta}{\sf TPCA}_{\alpha}$ has the interpolation property.
\item ${\sf TPCA}_{\alpha}$ has $SUPAP$.
\end{enumarab}
\end{theorem}
\begin{proof} The proof of the first item is like the proof of theorem \ref{in} undergoing the obvious modifications, 
and the second item follows from the first using
the reasoning in theorem \ref{s}.
\end{proof}
Note that the representability of any $\A\in \sf TPCA_{\alpha}$ can be easily proved using a simpler version of 
the above technique, where only one Henkin ultrafilter   
is needed to establish representability. 

Now do we have an omitting types theorem in the context of $\sf TPCA_{\alpha}$? The question itself is problematic because the signature
of $\sf TPCA_{\alpha}$ is necessarily uncountable 
even if the dimension is countable because in this case we have continuum many substitution operators, 
and it is known that  for `omitting types theorems' tied so much to
the Baire category theorem, countability is essential. But here we have the related notion to omitting types that can be approached 
in our new context, namely, that of {\it complete representability},
summarized in the following question:
If $\A$ is in $\sf TPCA_{\alpha}$,  is there a 
representation of $\A$ that preserves infinite meets and joins, whenever they exist? 

We make the notion of representation precise in our new context. 
$\B(X)$ denotes the Boolean set algebra 
$(\wp(X), \cup, \cap, \sim ,\emptyset)$. 

\begin{definition} A {\it representation} of $\A$ is a pair $(f, V)$ such that $V=\bigcup_{i \in I}{}^{\alpha}U_i$ for some indexing set 
$I$,  and a family $U_i: i\in I,$ of non-empty lsets that are pairwise disjoint, that is, 
 $U_i\cap U_j=\emptyset$ 
for distinct $i\neq j$, and $f:\A\to\langle \B(V), {\sf c}_{i}, I_i, {\sf s}_{\tau}\rangle_{i<\alpha, \Gamma\subseteq_{\omega} \alpha, 
\tau\in {}^{\alpha}{\alpha}}$
is an injective homomorphism, where $I_i$ as usual is defined by 
$I_i(X)=\{s\in V:  s_i\in {\sf int}\{u\in \bigcup_{s\in V} \rng s: s^i_u\in X\}\}, X\subseteq V.$
\end{definition}

In what follows we may drop the operations when talking about $\alpha$ dimensional 
set algebras (whose top elements consists of $\alpha$-ary sequences) 
identifying notationally 
the algebra with its universe. This does not cause any harm since set algebras are uniquely defined by their top element.

A completely additive Boolean algebra with operators 
is one for which all extra non-Boolean operations preserve arbitrary joins.

\begin{lemma}\label{r} Let $\A\in \sf TPCA_{\alpha}$. A representation $f$ of $\A$
is atomic if and only if it is complete. 
If $\A$ has a complete representation, then $\Rd_{pa}\A$ is completely additive.
\end{lemma}
\begin{demo}{Proof} The first part is like \cite{HH}. For the second part replace $\A$ by its complete representation 
where $\sum$ is $\bigcup$. 
It is clear that in such an algebra  the operations of substitutions and cylindrifiers 
are completely additive.
\end{demo}
By Lemma \ref{r} a necessary condition for the existence of complete representations is the condition of atomicity and complete
additivity of its $\sf CPA$ reduct.
We now prove the harder converse to this result, namely, that when $\A$ is atomic and $\Rd_{pa}\A$ 
is completely additive, then 
$\A$ is completely representable. We note that cylindrifiers are in all cases completely additive. 

\begin{theorem}\label{Ferenczi}  
Any atomic algebra in $\A\in \sf TPCA_{\alpha}$ such that $\Rd_{pa}\A$ is completely additive, 
is completely representable. 
\end{theorem}
\begin{proof}  Argument used is like the argument in \cite[Theorem 3.10 ]{conference} using a Henkin construction, expressed algebraically by 
dilating the algebra to large enough dimensions and then forming a Henkin ultrafilter (defined as before) 
in the dilation, with a very simple 
topological fact, namely, that in the Stone space of an 
atomic Boolean algebra principal ultrafilters lie outside sets of {\it the first category}; 
these are countable unions of no-where dense sets; so we could always find a principal Henkin ultrafilter from
which we build the complete representation.

Let $c\in \A$ be non-zero. We will find a set algebra $\B\in \sf TPCA_{\alpha}$ and a homomorphism 
from $f:\A \to \B$ that preserves arbitrary suprema
whenever they exist and also satisfies that  $f(c)\neq 0$. 
Now there exists $\B\in \sf TPCA_{\mathfrak{n}}$, $\mathfrak{n}$ a regular cardinal.
such that $\A\subseteq \Nr_{\alpha}\B$ and $A$ generates $\B$ and we can assume
that $|\mathfrak{n}\sim \alpha|=|\mathfrak{n}|$. We also have 
for all $Y\subseteq A$, we have $\Sg^{\A}Y=\Nr_{\alpha}\Sg^{\B}Y.$ 
This dilation also has Boolean reduct isomophic to $F({}^\mathfrak{n}\alpha, \A)$, in particular, it is atomic because $\A$ is atomic.
Cylindrifiers are defined on this minimal functional dilation exactly like in theorem \ref{dilation}
by restricting to singletions.
For all $i< \mathfrak{n}$, we have 
\begin{equation}\label{tarek1}
\begin{split}{\sf c}_{i}p=\sum {\sf s}_i^jp 
\end{split}
\end{equation}
This last supremum can be proved to hold using the same reasoning in \cite[Theorem 1.6]{DM}. 
Let $X$ be the set of atoms of $\A$. Since $\A$ is atomic, then  $\sum^{\A} X=1$. By $\A=\Nr_{\alpha}\B$, we also have $\sum^{\B}X=1$.
Because substitutions are completely additive, by assumption,  we have
for all $\tau\in {}^{\alpha}\mathfrak{n}$
\begin{equation}\label{tarek2}
\begin{split}
\sum {\sf s}_{\bar{\tau}}^{\B}X=1.
\end{split}
\end{equation}
Let $S$ be the Stone space of $\B$, whose underlying set consists of all Boolean ulltrafilters of 
$\B$. Let $X^*$ be the set of principal ultrafilters of $\B$ (those generated by the atoms).  
These are isolated points in the Stone topology, and they form a dense set in the Stone topology since $\B$ is atomic. 
So we have $X^*\cap T=\emptyset$ for every nowhere dense set $T$ (since principal ultrafilters, which are isolated points in the Stone topology,
lie outside nowhere dense sets). 
For $a\in \B$, let $N_a$ denote the set of all Boolean ultrafilters containing $a$.
Now  for all $\Gamma\subseteq \alpha$, $p\in B$ and $\tau\in {}^{\alpha}\mathfrak{n}$, we have,
by the suprema, evaluated in (1) and (2):
\begin{equation}\label{tarek3}
\begin{split}
G_{i,p}=N_{{\sf c}_{i}p}\sim \bigcup_{{\tau}\in {}^{\alpha}\mathfrak{n}} N_{s_{\bar{\tau}}p}
\end{split}
\end{equation}
and
\begin{equation}\label{tarek4}
\begin{split}
G_{X, \tau}=S\sim \bigcup_{x\in X}N_{s_{\bar{\tau}}x}.
\end{split}
\end{equation}
are nowhere dense. 
Let $F$ be a principal ultrafilter of $S$ containing $c$. 
This is possible since $\B$ is atomic, so there is an atom $x$ below $c$; just take the 
ultrafilter generated by $x$.
Then $F\in X^*$, so $F\notin G_{i, p}$, $F\notin G_{X,\tau},$ 
for every $i\in \alpha$, $p\in A$
and $\tau\in {}^{\alpha}\mathfrak{n}$.
Now define for $a\in A$
$$f(a)=\{\tau\in {}^{\alpha}\mathfrak{n}: {\sf s}_{\bar{\tau}}^{\B}a\in F\}.$$
Then $f$ is a homomorphism from $\A$ to the full
set algebra with unit $^{\alpha}\mathfrak{n}$, with interior operators $J(i)$ $i<\alpha$ defined 
exactly as in theorem \ref{in} using all substitutions instead of only finite ones.
We have  $f(c)\neq 0$ because $Id\in f(c)$. 
Moreover $f$ is an  atomic representation since $F\notin G_{X,\tau}$ for every $\tau\in {}^{\alpha}\mathfrak{n}$, 
which means that for every $\tau\in {}^{\alpha}\mathfrak{n}$, 
there exists $x\in X$, such that
${\sf s}_{\bar{\tau}}^{\B}x\in F$, and so $\bigcup_{x\in X}f(x)={}^{\alpha}\mathfrak{n}.$ 
We conclude that $f$ is a complete  representation.
\end{proof}

Now let ${\sf CTPCA}_{\alpha}$ be the class of completely representable $\sf TPCA_{\alpha}$s.
\begin{theorem} 
\begin{enumarab}
\item The class ${\sf CTPCA}_{\alpha}$ is elementary, and it is axiomatizable by a finite schema in first order logic.
\item Let $\Nr: \bold K\to \TPCA_{\alpha}$ be the neat reduct functor.
Then $\Nr$ is strongly invertible, namely, there is a functor $G:\TPCA_{\alpha}\to \bold K$ and natural isomorphisms
$\mu:1_{\bold K}\to G\circ \Nr$ and $\epsilon: \Nr\circ G\to 1_{\TPCA_{\alpha}}$
\end{enumarab}
\end{theorem}
\begin{proof} Atomicity can be expressed by a first order sentence, and complete additivity can 
be captured by the following continuum many formulas,
that form a single schema.
Let $\At(x)$ be the first order formula expressing that $x$ is an atom. That is $\At(x)$ is the formula
$x\neq 0\land (\forall y)(y\leq x\to y=0\lor y=x)$. For $\tau\in {}^{\alpha}\alpha$, let $\psi_{\tau}$ be the formula:
$$y\neq 0\to \exists x(\At(x)\land {\sf s}_{\tau}x\neq 0\land {\sf s}_{\tau}x\leq y).$$ 
Let $\Sigma$ be the set of first order formulas obtained  by adding all formulas $\psi_{\tau}$ $(\tau\in {}^{\alpha}\alpha)$
to the polyadic schema.
Then it is esay to show that ${\sf CTPCA}_{\alpha}={\bf Mod}(\Sigma)$. 
The second part follows by using exactly the same reasoning 
in \cite[Theorem 3.4]{conference}.
\end{proof}

We can also expand  the signature of the $\omega$ dimensional algebras studied in \cite{S, AUamal}, 
whose signature is countable having substitutions coming from a countable {\it rich} semigroup $G$,
by interior operators with the same equations postulated for $\sf TCA_{\alpha}$. 
Denote the resulting variety by $\sf TPA_G$.
Also usual Halmos  polyadic algebras can be enriched with such modalities, call the resulting variety
$\sf TPA_{\alpha}$. 
We get all positive results obtained for $\sf TPCA_{\alpha}$ with almost the same proofs using the techniques in 
\cite{AUamal, conference, super,S}.
In particular we have:
\begin{theorem}\label{polyadicsupap}
\begin{enumarab}
\item $\sf TPA_G$ and $\sf TPA_{\alpha}$ have the super amalgamation property
\item In each such variety $\A$ is completely representable 
if and only if the reduct obtained by discarding interior operators
is completely additive.
\item For such varieties the functor $\Nr$ (defined like before adapted to the present context)  is strongly invertible.
\end{enumarab}
\end{theorem}

We can add diagonal elements and {\it  relativize semantics} of topological polyadic algebras, getting 
the variety  of {\it topological cylindric polyadic algebras of dimension $\alpha$} 
whose signature is like $\sf TPCA_{\alpha}$ axiomatized by
the set of equations postulated  in \cite[definition 6.3.7]{Fer} together with the schema of equations for the interior operators.

Denote this abstract class by $\sf TPCEA_{\alpha}$ and the concrete class of representable algebras 
by $\sf TGp_{\alpha}$. Then using the same methods adopted her replacing Henkin ultrafilters by 
what Ferenzci calls {\it perfect ultrafilters} we get:
\begin{theorem}\label{f} Let $\alpha\geq \omega$. Then the following hold:
\begin{enumarab}
\item $\sf TGp_{\alpha}=\sf TPCEA_{\alpha}$; hence
$\sf TGp_{\alpha}$  is a variety that can be axiomatizable by a finite schema of Sahlqvist equations. Furthermore, it is canonical and 
atom-canonical.
\item $\sf TGp_{\alpha}$ has the superamalgamation property
\item Any atomic algebra in ${\sf TGp}_{\alpha}$ has a complete representation. In particular, 
the class of completely representable  $\sf TCPEA_{\alpha}$s  is elementary.
\end{enumarab}
\end{theorem}
\begin{demo}{Sketch} Suppose $\A$ is such an algebra.
Then a dilation can be formed using all available substitititions, so we get $\A=\Nr_{\alpha}\B$ where $\B\in \sf TCPEA_{\beta}$.
However in the process of representation only {\it admissable substitution on $\beta$} are used. A substituition $\tau\in {}^{\beta}\beta$ 
is such if $\dom \tau\subseteq \alpha$
and $\rng\tau\cap \alpha=\emptyset$. 
Call the set of all admissable substitutions $\sf adm$. Henkin ultrafilters can always be found,
but they are modified to give {\it perfect  ultrafilters} in the sense of \cite[p.128]{Sayedneat}. 

To preserve diaginal
elements one factors out the set $\Gamma=\{i\in \beta:\exists j\in \alpha: {\sf c}_i{\sf d}_{ij}\in G\}$ by the congruence
relation $k\sim l$ iff ${\sf d}_{kl}\in G.$ Then $\Gamma\subseteq \beta$
and the required representation with  base  $\Gamma/\sim$ is defined via  
$$f(a)=\{\bar{\tau}\in {}^{\alpha}[\Gamma/\sim]: \tau\in {\sf adm}, {\sf s}_{\tau}^{\B}a\in G\},$$
where for each $i\in \alpha$ and $\tau\in \sf adm$ 
$\bar{\tau}(i)=\tau(i)/\sim$, witness  \cite[p. 128]{Sayedneat}.
Next one defines for $p\in \A$ and $i\in \alpha$ the sets 
$O_{p,i}$ and the interior operators $J_i$ on 
the representation as before.

\end{demo}

\section{Summary of results on amalgamation and interpolation}

In the next table we summarize our results on classes of algebras
in tabular form. This task was done for different classes of cylindric algebras in the recent
\cite{MS}.

In Table 2, we summarize the results
we obtained on the interpolation property
for the free algebras corresponding to
the classes of algebras dealt with in Table 1.

In the top row of Table 1,  we find a list of nine different amalgamation and embedding
properties ($AP$ and $EP$ for short), together with the definability property $ES$  and at the leftmost column
we find a comprehensive list of classes of algebras occupying six rows.
At the top of the third column `strong $AP$ w.r.t rep.' means 'the strong 
amalgamation property with respect to the class of {\it representable algebras} in question',
while at the top of the fifth column `$AP$ w.r.t abs.' means the amalgamation property with respect to the class of {\it abstract algebras}. 
For example 
`strong $AP$ w.r.t to rep' is `strong $AP$ with respect to $\sf RTCA_{\alpha}$' and 
`$AP$ w.r.t abs.' is `$AP$ with respect to $\TCA_{\alpha}$.'

Table 2 contains a summary of the results
involving interpolation of the
dimension-restricted free algebras. The rows addressing semisimple, substitution, 
representable cylindric algebras in Table 1 collapse to just one row in Table 2, since the free algebra coincide for all these classes
(they all generate the same variety, namely, $\sf RTCA_{\alpha}$).

In  more general contexts than topological predicate logic addressed here,
the Craig interpolation property ramifies into several different interpolation properties ($IP$ for short).
These properties are summarized in the six
columns of the uppermost row of Table 2.

Only the first row in  the next table deals with finite dimensional algebras of dimension $n>1$.
$\sf K$ denotes any subclass of $\sf TCA_n$ containing the variety of representable algebras.
All the $\bf no$'s in this row follow readily from theorem \ref{Comer}.
\begin{table}[ht]
\caption{}
%\begin{minipage}[p] {10cm}
%\center
\hspace{-1 cm}
\begin{tabular}[b]{|l|c|c|c|c|c|c|c|c|c|c|}
     \hline
 &strong  &strong                     &$AP$             &$AP$                        &$AP$       &strong   &$EP$    &$EP$         &$SUP$    &$ES$               \\
 &$AP$    &$AP$                       &             &w.r.t                          &for             &$EP$      &           &   for           &$AP$       &             \\
    &     &w.r.t                         &                   &abs.   &simple                 &       &           &simple      &                 &            \\
   &     &rep.    &                     &                   &          &                &           &algebras   &                 &          \\
  \hline          
$\sf TRCA_n\subseteq \sf K$ &no&no&no&no&no&no&no&no&no&no\\

  \hline
         $\sf TLf_{\alpha}$ &yes&yes&yes&yes&yes&yes&yes&yes&yes&yes\\
   \hline
$\TDc_{\alpha}$ &no&yes&no&yes&no&no&no&no&no&yes\\
    \hline
     $\sf TSs_{\alpha}$ &no&no& yes&yes& yes&yes&yes&yes&no&no\\
       %$CA$'s&&&&&&&&&& \\
      \hline
     $\sf TSc_{\alpha}$ &no&no&yes&yes& yes&yes&yes&yes&no&no\\

%\hline
   %$\sf WSc_{\alpha}$ &?&?&?&?& yes&yes&yes&yes&no&no\\
                                                                                                             %$CA$'s&&&&&&&&&  \\
     \hline
     $\sf RTCA_{\alpha}$ &no&no   &no& no&yes&yes&yes&yes&no&no\\
  %&   &  &          &   &   &   &    &&   &\\
    %&                                                                                                       %$CA$'s&&&&&&&&&  \\
     \hline
     $S\Nr_{\alpha}\TCA_{\alpha+k}$
 &no&no   &no& no&yes&?&?&yes&no&no\\
  %&   &  &          &   &   &   &      &  &         &   &   &   &    &&   &\\
    \hline
     $\TCA_{\alpha}$ &no&no&no&no&no&yes&no&yes&no&no\\
                                                                                                         %$CA$'s&&&&&&&&&  \\
     \hline
     ${\sf TPCA}_{\alpha}$ &yes&yes   &yes& yes&yes&yes&yes&yes&yes&yes\\
  %&   &  &          &   &   &   &    &&   &\\
    %&   &  &         &   &   &   &    &&   &\\
    %\hline
     %${\sf MAc}_{\alpha}$ &no&no&no&no&yes&?&?&yes&no&no\\
                                                                                                             %$CA$'s&&&&&&&&&  \\
     \hline
     ${\sf TPA}_{G}$ &yes&yes   &yes& yes&yes&yes&yes&yes&yes&yes\\
  %&   &  &          &   &   &   &    &&                                                                        %
                                                                                                          %$CA$'s&&&&&&&&&  \\
     \hline
     ${\sf TPA}_{\alpha}$ &yes&yes   &yes& yes&yes&yes&yes&yes&yes&yes\\
  %&   &                                                                                                 %$CA$'s&&&&&&&&&  \\
     \hline
     ${\sf TGp}_{\alpha}$ &yes&yes   &yes& yes&yes&yes&yes&yes&yes&yes\\
  %&   &  &          &   &   &   &    &&           &          &   &   &   &    &&           
\hline
\end{tabular}
\end{table}
Now we clarify the results  collected above, stating where can they be found in the text.
The last four rows in Table 1 and Table 2  follow from theorems \ref{positive},  \ref{polyadicsupap}, \ref{f}. Next we have;

\begin{enumarab}

\item Row two: Here we are dealing with ordinary predicate topological logic which has $IP$
as proved in theorem \ref{in}.
The rest now follow from theorems \ref{SUPAP} and \ref{e}.

\item Row three: All the {\bf no}'s in second row follows from example \ref{d}. For simple algebras
the algebras $\A$ and $\B$ taken in example \ref{d} can be easily chosen to be simple. 
Though there is a {\it simple} amalgam, it is {\it not} dimension complemented.
As illustrated in example \ref{d}, there could not be one.

The first and second
{\bf yes} follow from theorem \ref{SUPAP}, and
the third {\bf yes} is due to the fact that $ES$
follows from the fact that $\Dc_{\alpha}$ has $SUPAP$ w.r.t $\sf RTCA_{\alpha}$, 
resorting to theorem \ref{e}.

\item Row four: $\sf TSs_{\alpha}$ does not have $SAP$ because it does not have $ES$, by theorem \ref{es}. 
In fact, this last theorem  takes care of all
the ${\sf no}$'s.
The remaining {\bf yes}'s follow from theorem \ref{sc}
and corollary \ref{embedding}.

\item Row five: The results follow like in the previous item.
In particular, the {\bf no}'s follow from theorems \ref{es} using theorem \ref{e}.

\item Row six: The ${\bf no}$'s except for the last 
follow from theorem \ref{am}; the last ${\bf no}$ follows from
theorem \ref{es}. 
The ${\bf yes}$'s follow from theorem \ref{sc}
and corollary \ref{embedding}. 

\item Row seven. Like row six, except that the various forms of $EP$ for algebras that are not simple 
remains unsettled. In theorem \ref{simon} the base algebra is the minimal subalgebra of an algebra obtained
by {\it twisting} a representable algebra, and the other algebra is representable. But twisted algebras do not satisfy the so-called merry go round 
identities, which algebras in $S\Nr_{\alpha}\CA_{\alpha+2}$ do. So this technique does not work for $S\Nr_{\alpha}\TCA_{\alpha+k}$ when
$k\geq 2$. 
\item Row eight. Note that $\TCA_{\alpha}$ does not have $AP$ 
with respect to $\sf TRCA_{\alpha}$ is trivial. One just takes a non-representable algebra 
$\A$ and considers the inclusion maps $i:\A\to \A$ twice, so that we are required to amalgamate $\A$ 
over $\A$ by a representable algebra, which is impossible for the amalgam necessarily contains an isomorphic copy of $\A$, while
any subalgebra of a representable algebra is representable. The ${\bf no}$'s follow from theorems \ref{am}, \ref{es} and \ref{simon},
and the only ${\bf yes}$  from theorem \ref{sc}.

%\item Row nine. All the ${\bf yes}$'s follow from theorem \ref{positive}.
\end{enumarab}

For $\sf TWSc_{\alpha}$  all questions involving $AP$ remains unsettled. If any of the conditions in theorem \ref{twsc} 
hold, then we get a  {\bf no} for all such questions, for in this case we get that 
$\sf TWSc_{\alpha}=\sf RTCA_{\alpha}$.

In the following table $IP$ is short for interpolation property. The top row addresses
all interpolation properties introduced and investigated throughout  this paper. $IP$ is the interpolation property, $WIP$ is weak $IP$,
$AIP$ is almost $IP$, \ldots etc. The first column addresses $\sf K$ where $\sf K$ is any class 
between $\sf TRCA_n$ and $\sf TCA_n$ $n$ is finite $>1$. The {\bf no}'s in this row follows 
from theorem \ref{Comer} and corollary \ref{notequivalent}. Wiuthout loss of generality, 
we consider (countable) 
free algebras on $\omega$ generators.

Notice that $IP$ implies $AIP$ implies
$WIP$, and strong restricted $IP$ implies all other restricted versions of $IP$.
\begin{table}[ht]
\caption{}
%\begin{minipage}[p] {10cm}
%\center
\begin{tabular}{|l|c|c|c|c|c|c|}    \hline
						&$IP$  &  $ AIP$                         &            &                         &almost restricted  		 &weak                         \\
						&    &                        &$WIP$	     &restricted                   &$IP$	                        &restricted                     \\
						&	  &                        &                      &$IP$                           &	  &                                  $IP$	                    \\

                                                               \hline
                                                                          $\Fr_{\omega}{\sf K}$&no&no&no&no&no&no\\

						\hline
                                                                        $\Fr_{\omega}^{\rho}\TCA_{\alpha}$&yes&yes&yes&yes&yes&yes\\
                                                                              which is&&&&&& \\
                                                                               in $\sf TLf_{\alpha}$&&&&&&\\
                                                                             \hline
                                                                         $\Fr_{\omega}^{\rho}\TCA_{\alpha}$
                                                                          &yes&yes&yes&yes&yes&yes\\
                                                                              which is&&&&&& \\

                                                                         in $\sf TDc_{\alpha}$&&&&&&\\
                                                                          %\hline $\Fr_{\omega}^{\rho}\sf SA_{\alpha}$
                                                                          %&yes&yes&yes&yes&yes&yes\\
                                                                             % which is&&&&&& \\
                                                                               %in $\Dc_{\alpha}$&&&&&&\\

                                                                           \hline
                                                                          $\Fr_{\omega}\sf RTCA_{\alpha}$&no&no&yes&yes&yes&yes\\
                                                                             %&&&&&&\\
                                                                            %\hline
                                                                           \hline
                                                                          $\Fr_{\omega}\sf S\Nr_{\alpha}\sf TCA_{\alpha+k}$&no&no&yes&?&?&?\\
                                                                             %&&&&&
                                                                           \hline
                                                                          %$\Fr_{\omega}{\sf RMAc}_{\alpha}$&no&yes&no&yes&yes&yes\\
                                                                             %&&&&&&\\

                                                                           %\hline
                                                                          $\Fr_{\omega}{\sf TCA}_{\alpha}$&no&yes&no&no&no&no\\
                                                                            
                                                               \hline
                                                                          $\Fr_{\omega}{\sf TPCA}_{\alpha}$&yes&yes&yes&yes&yes&yes\\
                                                                                
                                                               \hline
                                                                          $\Fr_{\omega}{\sf TPG}$&yes&yes&yes&yes&yes&yes\\

                                                               \hline
                                                                          $\Fr_{\omega}{\sf TPA}_{\alpha}$&yes&yes&yes&yes&yes&yes\\

                                                               \hline
                                                                          $\Fr_{\omega}{\sf TGp}_{\alpha}$&yes&yes&yes&yes&yes&yes\\

\hline

\end{tabular}
\end{table}

\vspace{10cm}

All positive theorems on the free algebras, addressing cases other than $\sf RTCA_{\alpha}$
follow from theorems \ref{in} and \ref{positive}.

Concerning the $\sf RTCA_{\alpha}$ case,
the positive results follow from  theorems \ref{embedding}, \ref{weak2} 
and the negative results follow from theorems 
\ref{am} and \ref{simon}. The various forms of restricted $IP$ remain unsettled 
for $\Fr_{\omega}S\Nr_{\alpha}\TCA_{\alpha+k}$.


\begin{thebibliography}{99}

%\bibitem{Andreka} H. Andr\'eka, {\it Complexity of equations valid in algebras of relations.} Ann Pure and App Logic {\bf 89}(1997) 
%p. 149-209.

\bibitem{ACNS} Andr\'eka, H. Comer, C.,  Madar\'asz, J., N\'emeti,I. and Sayed Ahmed, T. 
{\it Epimorphisms in cylindric algebras.}
Algebra Universalis, 61 3-4 261-281 2009



\bibitem{1} H. Andr\'eka, M. Ferenczi and  I. N\'emeti (Editors), {\bf Cylindric-like Algebras and Algebraic Logic},
Bolyai Society Mathematical Studies and Springer-Verlag, {\bf 22} (2012).

%\bibitem{AGN}H. Andr\'eka, T. Gregely H.  and I. N\'emeti, 
%{\it On universal algebraic constructions of logics.} 
%Studia Logica , {\bf 36} (1977) p.9-47.

%\bibitem{AMN}H. Andr\'eka, J.D. Monk., I. N\'emeti,I. (editors) {\it Algebraic Logic,}
%North- Holland, Amsterdam, 1991.


%\bibitem{AN75} Andr\'eka,H., N\'emeti,I., 
%{\it A simple purely algebraic proof of the completeness of some first order logics.}
%Algebra Universalis, {\bf 5} (1975) p.8-15.


%\bibitem{AN80} Andr\'eka,H.,Nemeti,I., {\it On Systems of varieties definable by schemes of equations.} 
%Algebra Universalis, {\bf 11}(1980) p. 105-116.

%\bibitem{ANS} Andr\'eka, H.,  N\'emeti, I., Sain, I {\it Algebraic Logic} (2000).
%In Handbook of Philosophical Logic, Editors Gabbay et all.

%\bibitem{ANT} H. Andr\'eka, I. N\'emeti and T. Sayed Ahmed, 
%{\it Omitting types for finite variable fragments and complete representations.}
%Journal of Symbolic Logic {\bf 73} (2008) p. 65-89 

%\bibitem{BP} Blok,W.J., and Pigozzi,D. {\it Algebraizable logics}. 
%Memoirs of American  Mathematical Society, {\bf 77}(396), 1989.

\bibitem{Comer} Comer S.D. {\it Classes without the amalgamation property} 
Pacific journal of Mathematics, {\bf 28}(2)(1969), p.309-318. 



%\bibitem{Chang} Chang {\it Modal model theory.} Proceedings of the Cambridge Summer School in mathematical logic, Lecture Notes 
%{\bf 337} (1974) p. 599-617.

\bibitem{D}A. Daigneault, {\it Freedom in polyadic algebras and two theorems
of Beth and Craig}. Michigan Math. J. {\bf 11}(1963), p. 129-135.

\bibitem{DM} A. Daigneault and J.D. Monk,
{\it Representation Theory for Polyadic algebras}.
Fund. Math. {\bf 52}(1963), p.151-176.


\bibitem{Fer} M. Ferenczi,
{\it A new representation theory: Representing cylindric-like algebras by relativized set algebras}, in \cite{1} p. 135-162.



%\bibitem{g4} G.Georgescu {\it A representation theorem for tense polyadic algebras.} Mathematiuca, Tome 21 {\bf 44} (2) 
%(1979) p.131-138.

%\bibitem{g3} G. Georgescu {\it Modal polyadic algebras.} Bull. Math Soc. Sci Math R,S Romaina (1979) {\bf 23} p.49-64


%\bibitem{g} G. Georgescu {\it Algebraic analysis of topological logic.} Mathematical Logic Quarterly (28) p.447-454 (1982) 
%{\bf 52}(5)(2006) p.44-49.

%\bibitem{g5} G. Georgescu, {\it A representation theorem for polyadic Heyting algebras.}
%Algebra Universalis, {\bf 14} (1982) , p.197-209.

%\bibitem{g2} G. Georgescu {\it Chang's modal operators in Algebraic Logic.} Studia Logica {\bf 42}(1), (1983) p.43-48


%\bibitem{Halmos} P. Halmos, {\it Algebraic Logic.}
%Chelsea Publishing Co., New York, (1962.)


%\bibitem{Henkin} L. Henkin, {\it An extension of the Craig-Lyndon interpolation theorem.} 
%Journal of Symbolic Logic {\bf 28}(3) (1963), p.201-216

\bibitem{HMT1}L. Henkin, J.D. Monk  and A.Tarski,  {\it Cylindric Algebras Part I}.
North Holland, 1971.

\bibitem{HMT2} L. Henkin, J.D. Monk and A. Tarski, {\it Cylindric Algebras Part II}.
North Holland, 1985.


\bibitem{cat}H. Herrlich, G. Strecker, {\it Category theory.} Allyn and Bacon, Inc, Boston (1973)

%\bibitem{r} R. Hirsch, {\it Relation algebra reducts of cylindric algebras and complete representations},
%Journal of Symbolic Logic, {\bf 72}(2) (2007), p.673-703.



\bibitem{HH} Hirsch R., Hodkinson, I. {\it Complete representations in algebraic logic}
Journal of Symbolic Logic, 62, 3 (1997) 816-847 


%\bibitem{HHbook} R. Hirsch and I. Hodkinson, {\it Relation algebras by games.}
%Studies in Logic and the Foundations of Mathematics, volume {\bf 147} (2002)

%\bibitem {HHbook2} R. Hirsch and I. Hodkinson, {\it Completions and complete representations in algebraic logic.} In \cite{1}

%\bibitem{t} R. Hirsch and T. Sayed Ahmed, {\it The neat embedding problem for algebras other than cylindric algebras
%and for infinite dimensions.} Journal of Symbolic Logic (in press).


%\bibitem{Hodkinson} I. Hodkinson, {\it Atom structures of relation and cylindric algebras}. Annals of pure and applied logic,
%{\bf 89}(1997), p.117-148.

%\bibitem {AU} I. Hodkinson, \emph{A construction of cylindric and polyadic algebras from atomic relation algebras.}
%Algebra Universalis, {\bf 68} (2012), p. 257-285.

%\bibitem{k} A.S. Kechris, \emph{Classical Descriptive Set Theory}, Springer Verlag, New York (1995).


%\bibitem{k} M. Khaled and T. Sayed Ahmed, {\it Classes of Algebras not closed under completions.}
%Bulletin Section of Logic {\bf 38}(2009) p. 29-43.

\bibitem{Mad} J. Madar\'asz
{\it Interpolation and Amalgamation;  Pushing the Limits. Part I}
Studia Logica, {\bf 61}, (1998) p. 316-345.

%\bibitem{Mad98} Madar\'asz, J. {\it interpolation and Amalgamation; 
%Pushing the Limits. Part II} Studia Logica.

\bibitem{AUU} J. Mad\'arasz and T. Sayed Ahmed,
{\it Amalgamation, interpolation and epimorphisms.}
Algebra Universalis {\bf 56} (2) (2007), p. 179 - 210, 

\bibitem{MStwo} J. Mad\'arasz and T. Sayed Ahmed,
{\it Neat reducts and amalgamation in retrospect, a survey of results and some
methods. Part 2: Results on amalgamation.} Logic Journal of IGPL {\bf 17}, (2009), p.755-802.

\bibitem{MS} J. Mad\'arasz and T. Sayed Ahmed
{\it Amalgamation, interpolation and epimorphisms in algebraic logic}.  In \cite{1}, p.91-104

%\bibitem{z} Makowski and Ziegler {\it Topological model theory with an interior operator.} Preprint


\bibitem{Mak} L. Maksimova 
{\it Amalgamation and interpolation in normal modal logics}.
Studia Logica {\bf 50}(1991) p.457-471. 

%\bibitem{z2} J. Marowski  and A. Marcia {\it Completeness theorem for modal model theory with Montague Chang semantics.} 
%This Zeisachr {\bf  23} (1977) 97-104.



\bibitem{george} G. Metcalfe, F. Montagna , and C. TsiNakis {\it Amalgamation and interpolation in ordered algebras.} 
(2012) pre-print.



\bibitem{P} D. Pigozzi,
{\it Amalgamation, congruence extension, and interpolation properties in algebras.}
Algebra Universalis.
{\bf 1}(1971), p.269-349.



%\bibitem{Sagi} G. S\'agi and D. Szir\'aki, \emph{Some Variants of Vaught's Conjecture from the Perspective of Algebraic Logic}, 
%Logic Journal of the IGPL, published online January 5, 2012.



\bibitem{S} I. Sain {\it Searching  for a finitizable algebraization of first order logic}. 
8, Logic Journal of IGPL. Oxford University, Press (2000) no 4, p.495--589.

%\bibitem{ST} I. Sain and R. Thompson,
%{\it Strictly finite schema axiomatization of quasi-polyadic algebras}, in
%{\bf Algebraic Logic} H. Andr\'eka, J. D. Monk and I. N\'emeti (Editors),  North Holland (1990) p.539-571.


%\bibitem{IGPL2}T. Sayed Ahmed, {\it The class of neat reducts is not elementary.} Logic Journal of $IGPL$, {\bf 9}(2001)p. 593-628.

%\bibitem{MLQ} T. Sayed Ahmed, {\it A model-theoretic solution to a problem of Tarski.} Math Logic Quarterly, Vol. 48 (2002), pp. 343-355.

\bibitem{AUamal} T. Sayed Ahmed, {\it On Amalgamation of Reducts of Polyadic Algebras.}
Algebra Universalis  {\bf 51} (2004), p.301-359.

\bibitem{IGPL} T. Sayed Ahmed,
{\it An interpolation theorem for first order logic with infinitary predicates.}
Logic journal of IGPL, {\bf 15}(1) (2007), p.21-32

%\bibitem{weak} T. Sayed Ahmed, {\it Weakly representable atom structures that are not strongly representable,
%with an application to first order logic,} Mathematical Logic Quarterly, {\bf 54}(3)(2008) p. 294-306.

%\bibitem{ex} Sayed Ahmed, T., {\it On finite axiomatizability of expansions of cylindric algebras.} 
%Journal of Algebra, Number Theory, Advances and Applications, {\bf 1}(2010),  p.19-40.

\bibitem{super} T. Sayed Ahmed, {\it The class of polyadic algebras has the super amalgamation property}
Mathematical Logic Quarterly {\bf 56}(1)(2010), p.103-112



\bibitem{amal} T. Sayed Ahmed,  {\it Classes of algebras without the amalgamation property.} Logic Journal of IGPL,
{\bf 19} (2011), p. 87-104.

\bibitem{Hung} T. Sayed Ahmed {\it Representability and amalgamation in Heyting polyadic algebras}, Studia Mathematica
Hungarica, {\bf 48}(4)(2011), p. 509-539.

\bibitem{univl} T. Sayed Ahmed, {\it Amalgamation in universal algebraic logic.} Stududia Mathematica Hungarica,
{\bf 49} (1) (2012), p. 26-43.

\bibitem{typeless} T. Sayed Ahmed, {\it Three interpolation theorems for typeless logics.} Logic Journal of $IGPL$ {\bf 20}(6)
(2012), p.1001-1037.

\bibitem{es} T. Sayed Ahmed, {\it Epimorphisms are not surjective even in simple algebras.}
Logic Journal of $IGPL$, {\bf 20}(1) (2012), p. 22-26.


\bibitem{conference} T. Sayed Ahmed, {\it Neat embeddings as adjoint situations.}
Published on Line, Synthese. DOI 10.1007/s11229-013-0344-7.

\bibitem{Sayed} T. Sayed Ahmed {\it Completions, complete representations and omitting types.} In \cite{1}.

\bibitem{Sayedneat} T. Sayed Ahmed, {\it Neat reducts and neat embeddings in cylindric algebras.} In \cite{1}.




\bibitem{logica} T. Sayed Ahmed 
{\it Results on neat embeddings with applications to algebraizable extensions of first order 
logic.} Submitted.


\bibitem{part1} T. Sayed Ahmed
{\it An algebraic approach to topological logic and Chang's modal logic 
using cylindric algebras, Part 1: {\bf Topological logic via cylindric algebras.}} 

\bibitem{part3} T. Sayed Ahmed {\it An algebraic approach to topological logic and Chang's modal logic 
using cylindric algebras, Part 3\\
{\bf Some more algebra; finite dimensional topological algebras }}


\bibitem{part4} T. Sayed Ahmed {\it An algebraic approach to topological logic and Chang's modal logic 
using cylindric algebras, Parts 3-4:
{\bf Logical consequences. }}


\bibitem{logica} T. Sayed Ahmed 
{\it Results on neat embeddings with applications to algebraizable extensions of first order 
logic.} Submitted.

%\bibitem{Shelah} S. Shelah, {\it Classification theory: and the number of non-isomorphic models}
%Studies in Logic and the Foundations of Mathematics.  (1990).


\bibitem{Simon} Simon {\it Non-representable algebras of relations.} Ph.D dissertation, 
Mathematical institute, of the Hungarian Academy of Sciences (1997).



%\bibitem{s2} Sgro {\it The interior operator logic and product topologies.} Trans. Amer. Math. Soc. {\bf 258}(1980) p. 99-112.



%\bibitem{s} Sgro {\it Completeness theorems for topological models.} Annals of Mathematical Logic (1977) p.173-193.

\end{thebibliography}
\end{document}